\documentclass[twoside,centertags                 
]{amsart}

\usepackage[cmtip,color,all,curve,
arc,poly
]{xy}

\usepackage{amssymb}
\usepackage{graphicx}
\usepackage{mathrsfs}

\newtheorem{thm}[equation]{Theorem}
\newtheorem{cor}[equation]{Corollary}
\newtheorem{lem}[equation]{Lemma}
\newtheorem{prop}[equation]{Proposition}
\theoremstyle{definition}
\newtheorem{defn}[equation]{Definition}
\theoremstyle{remark}
\newtheorem{rem}[equation]{Remark}
\newtheorem{exm}[equation]{Example}

\newcommand{\C}[1]{\mathscr{#1}}

\newcommand{\abs}[1]{|#1|}

\def\op{\mathrm{op}}
\newcommand{\operad}[1]{\operatorname{Op}(#1)}

\newcommand{\algebra}[2]{\operatorname{Alg}_{#2}(#1)}
\newcommand{\algebrau}[2]{\underline{\operatorname{Alg}}_{#2}(#1)}
\newcommand{\algebrapc}[2]{\mathrm{Alg}^{pc}_{#2}(#1)}

\def\r{\rightarrow} 

\def\into{\rightarrowtail}
\def\onto{\twoheadrightarrow}

\newcommand{\deltan}[1]{\mathbf{#1}}

\newcommand{\id}[1]{\mathrm{id}_{#1}}

\newcommand{\Mod}[1]{\mathrm{Mod}(#1)}
\newcommand{\Modpc}[1]{\mathrm{Mod}^{pc}(#1)}

\def\hom{\operatorname{Hom}}
\def\rmap{\operatorname{Map}}
\def\rmapu{\underline{\operatorname{Map}}}

\def\rspecu{\mathbb{R}\underline{\operatorname{Spec}}}
\def\spec{\operatorname{Spec}}

\def\qcoh{\underline{\operatorname{QCoh}}}
\def\perf{\underline{\operatorname{Perf}}}

\def\ho{\operatorname{Ho}}

\def\aut{\operatorname{Aut}}

\def\aff{{\operatorname{Aff}}}
\def\comm{{\operatorname{Comm}}}

\def\st{\stackrel} 

\def\ul{\underline} 

\def\unit{\mathbb{I}} 

\def\To{\longrightarrow}

\def\colim{\mathop{\operatorname{colim}}}

\numberwithin{equation}{section}

\newdir{ >}{{}*!/-7pt/@{>}}
\newdir{> }{{}*!/10pt/@{>}}
\newdir{>> }{{}*!/10pt/@{>>}}

\begin{document}

\title
{Moduli spaces of algebras over non-symmetric operads 
}%
\author{Fernando Muro}%
\address{Universidad de Sevilla,
Facultad de Matem\'aticas,
Departamento de \'Algebra,
Avda. Reina Mercedes s/n,
41012 Sevilla, Spain\newline
\indent \textit{Home page}: \textnormal{\texttt{http://personal.us.es/fmuro}}}
\email{fmuro@us.es}

\subjclass[2010]{18D50, 55U35, 14K10}
\keywords{operad, algebra, associative algebra, unital algebra, model category, mapping space, moduli stack}

\begin{abstract}
In this paper we study spaces of algebras over an operad (non-symmetric) in symmetric monoidal model categories. We first compute the homotopy fiber of the forgetful functor sending an algebra to its underlying object, extending a result of Rezk. We then apply this computation to the construction of geometric moduli stacks of algebras over an operad in a homotopical algebraic geometry context in the sense of To\"en and Vezzosi. We show under mild hypotheses that the moduli stack of unital associative algebras is a Zariski open substack of the moduli stack of non-necessarily unital associative algebras.  The classical analogue for finite-dimensional vector spaces was noticed by Gabriel. 
\end{abstract}

\maketitle


\section{Introduction}

Let $\Bbbk$ be a commutative ring. An associative algebra structure on a free $\Bbbk$-module $F$ of rank $n$ with basis $\{e_{1},\dots,e_{n}\}\subset F\cong \Bbbk^{n}$ is determined by structure constants $c^{k}_{ij}$, $1\leq i,j,k\leq n$, such that
$$e_{i}\cdot e_{j}=\sum_{k=1}^{n}c_{ij}^{k}e_{k}.$$
The associativity condition $(e_{i}\cdot e_{j})\cdot e_{k}=e_{i}\cdot (e_{j}\cdot e_{k})$ is equivalent to the following identities of structure constants,
\begin{align*}
\sum_{m=1}^{n} c^{m}_{ij}c_{mk}^{l}=\sum_{m=1}^{n} c_{im}^{l}c_{jk}^{m},\qquad 1\leq l\leq n.
\end{align*}
Therefore, the moduli space of associative algebra structures on a free module of rank $n$ is the finitely presented affine subspace $\spec R\subset\mathbb \mathbb{A}^{n^{3}}$,
\begin{align*}
R&=\Bbbk[c_{ij}^{k}\,;\, 1\leq i,j,k\leq n]\left/\left(
\sum_{m=1}^{n} \left(c^{m}_{ij}c_{mk}^{l}- c_{im}^{l}c_{jk}^{m}\right)\,;\,1\leq i,j,k,l\leq n
\right)\right..
\end{align*}

A unital associative algebra structure on $F$ is given by an associative algebra structure together with a unit element $$\mathbf{1}=\sum_{i=1}^{n} a_{i}e_{i}\in F$$  satisfying the following equations, $1\leq k\leq n$, which are equivalent to  $\mathbf{1}\cdot e_{i}=e_{i}=e_{i}\cdot\mathbf{1}$, $1\leq i\leq n$,
\begin{align*}
(a_{1},\dots,a_{n})\left(
\begin{array}{ccc}
c_{11}^{k}&\cdots&c_{1n}^{k}\\
\vdots&\ddots&\vdots\\
c_{n1}^{k}&\cdots&c_{nn}^{k}
\end{array}
\right)&=\;\st{\begin{array}{c}\scriptstyle k^{\text{th}}\text{ place}\\\downarrow\end{array}}{(0,\dots,1,\dots,0)},\\
\left(
\begin{array}{ccc}
c_{11}^{k}&\cdots&c_{1n}^{k}\\
\vdots&\ddots&\vdots\\
c_{n1}^{k}&\cdots&c_{nn}^{k}
\end{array}
\right)
\left(\!\!\begin{array}{c}
a_{1}\\\vdots\\a_{n}
\end{array}\!\!\right)
&=
\left(\!\!\begin{array}{c}
0\\\vdots\\1\\\vdots\\0
\end{array}\!\!\right)
\leftarrow\;\scriptstyle k^{\text{th}}\text{ place}\,\displaystyle.
\end{align*}
Hence, the moduli space of unital associative algebra structures on a free module of rank $n$ is the affine subspace $\spec S\subset\mathbb{A}^{n^{3}+n}$,
\begin{align*}
S&=R[a_{1},\dots,a_{n}]\left/\left(
\delta_{jk}-\sum_{i=1}^{n} a_{i}c_{ij}^{k},\;\delta_{jk}-\sum_{i=1}^{n} c_{ji}^{k}a_{i}
\;;\; 1\leq j,k\leq n
\right)\right.,
\end{align*}
where $\delta_{kk}=1$ and $\delta_{jk}=0$ if $j\neq k$.

The morphism $f\colon \spec S\rightarrow \spec R$ consisting of forgetting the unit is induced by the inclusion $R\subset S$. This morphism is a categorical monomorphism since an associative algebra may have at most one unit. Moreover, $f$ is a Zariski open immersion, i.e.~it is also flat and of finite presentation, compare \cite[2.1 Lemma]{frto} and \cite[page 4]{gra}. 

Associative algebra structures are somewhat rigid. We are rather interested in them up to isomorphism. The algebraic group $\operatorname{GL}_{n}\cong\aut_{\Bbbk}(F)$ acts on $\spec R$. The orbits are the isomorphism classes of associative algebra structures. The isotropy group at a given point is the automorphism group of the corresponding associative algebra structure. 

In order to obtain a meaningful quotient which remembers all this, we must move to the category of algebraic stacks. The quotient stack
$$\ul{\operatorname{Ass}}_{n}=\spec R/\operatorname{GL}_{n}$$
is the moduli space of associative algebras on rank~$n$ vector bundles. The same applies to $\spec S$, and the quotient 
$$\ul{\operatorname{uAss}}_{n}=\spec S/\operatorname{GL}_{n}$$
is the moduli stack of unital associative algebras on rank~$n$ vector bundles. The morphism $f\colon \spec S\rightarrow \spec R$ induces a morphism between these algebraic stacks
$$\ul{f}\colon\ul{\operatorname{uAss}}_{n}\To \ul{\operatorname{Ass}}_{n}$$
which inherits most properties from $f$, e.g.~$\ul{f}$ is a monomorphism, affine, locally of finite presentation, and
flat, compare  \cite{cha}. 

These stacks can be described as follows. The category $\aff_\Bbbk$ of affine schemes over $\Bbbk$ is opposite to the category  of commutative (associative and unital) $\Bbbk$-algebras. We assume this category is endowed with the \'etale topology. For any commutative  $\Bbbk$-algebra $A$, let $\operatorname{Ass}_n(A)$ be the category of associative $A$-algebras whose underlying $A$-module is locally free of rank $n$. Denote $i\operatorname{Ass}_n(A)$ the subcategory of isomorphisms. Change of coefficient functors give rise to a pseudo-functor
\begin{align*}
\aff_\Bbbk^\op&\To\operatorname{Groupoids},\\
A&\;\mapsto\;i\operatorname{Ass}_n(A),
\end{align*}
which is the `functor of points' of the stack $\ul{\operatorname{Ass}}_n$. Similarly, if $\operatorname{uAss}_n(A)$ is the category of unital associative $A$-algebras whose underlying $A$-module is locally free of rank $n$, the `functor of points' of $\ul{\operatorname{uAss}}_n$ is
\begin{align*}
\aff_\Bbbk^\op&\To\operatorname{Groupoids},\\
A&\;\mapsto\;i\operatorname{uAss}_n(A),
\end{align*}
and the morphism $\ul{f}\colon\ul{\operatorname{uAss}}_{n}\rightarrow \ul{\operatorname{Ass}}_{n}$ is induced by  the functors forgetting the unit $\operatorname{uAss}_n(A)\r\operatorname{Ass}_n(A)$.

In this paper we consider the same situation in a homotopical algebraic geometry (HAG) context in the sense of To\"en and Vezzosi~\cite{hagII}. In such a context, the category of $\Bbbk$-modules is replaced with a symmetric monoidal model category $\C V$ and the category of affine schemes $\aff_{\C V}$ is the opposite of the category of commutative algebras (i.e.~monoids) in $\C V$. There is also a fixed model pre-topology $\tau$ and a class $\mathbf P$ of morphisms which plays the role of smooth morphisms.  A stack is a contravariant functor from affine schemes to simplicial sets, $\aff_{\C V}^{\op}\r\operatorname{Set}^{\Delta^{\op}}$, satisfying certain homotopy invariance and descent properties.

Given an operad $\mathcal O$  in $\C V$ (non-symmetric), admissible in the sense of Definition \ref{admissible}, we define a stack  $\algebrau{\mathcal O}{\C V}$ such that, for any commutative algebra $A$, 
$$\algebrau{\mathcal O}{\C V}(A)\simeq \abs{w\algebra{\mathcal O}{\Mod{A}}}$$
is the classifying space of the category of weak equivalences between $\mathcal O$-algebras in the category $\Mod{A}$ of $A$-modules. We call $\algebrau{\mathcal O}{\C V}$  the \emph{moduli stack of $\mathcal O$-algebras}. 

Notice that we proceed in a way inverse to de classical situation illustrated above. We do not present $\algebrau{\mathcal O}{\C V}$ as a quotient of an affine stack, but in terms of its `functor of points'. The connection to affine stacks is  described below.

Asmissibility is not a very strong condition since all cofibrant operads are admissible. Moreover, the associative and unital associative operads, $\mathtt{Ass}$  and $\mathtt{uAss}$, are always admissible. 

The stack of quasi-coherent modules is a stack $\qcoh$ such that, for any commutative algebra $A$, 
$$\qcoh(A)\simeq \abs{w\Mod{A}}$$
is the classifying space of the category of weak equivalences between $A$-modules. We show that the forgetful functors from $\mathcal O$-algebras to modules give rise to a morphism of stacks
$$\xi^{\mathcal{O}}\colon \algebrau{\mathcal{O}}{\C V}\To\qcoh.$$
The stack $\qcoh$ is too big and one is often interested in smaller substacks satisfying nice geometric properties, such as the stack $\ul{\operatorname{Vect}}_{n}$ of rank $n$ vector bundles, or the stack $\ul{\operatorname{Perf}}$ of perfect modules, see \cite[\S1.3.7]{hagII}. We will consider a generic substack $F\subset\qcoh$ such that the connected components of $F(A)$ are represented by perfect $A$-modules, and often restrict to the substack $\algebrau{\mathcal{O}}{F}\subset \algebrau{\mathcal{O}}{\C V}$ obtained as the homotopy pull-back of $\xi^{\mathcal O}$ along the inclusion $F\subset\qcoh$.

We prove that the restriction 
$$\xi^{\mathcal{O}}_{F}\colon \algebrau{\mathcal{O}}{F}\To F$$ 
is an affine morphism (Theorem \ref{lageo}). The homotopy fiber of $\xi^{\mathcal O}$ at an $A$-point $\rspecu(A)\r F$ represented by a perfect $A$-module $M$ is an affine stack 
$$\rmapu_{\operad{\C V}}(\mathcal{O},\mathtt{End}_{\Mod{A}}(M))$$
called \emph{moduli stack of $\mathcal O$-algebra structures on $M$}, since, for any $A$-algebra $B$, 
$$\rmapu_{\operad{\C V}}(\mathcal{O},\mathtt{End}_{\Mod{A}}(M))(B)\simeq
\rmap_{\operad{\C V}}(\mathcal{O},\mathtt{End}_{\Mod{B}}(\widetilde{M\otimes_{A}^{\mathbb L}B}))$$
is the mapping space in the model category $\operad{\C V}$ of operads in $\C V$ from $\mathcal O$ to the endomorphism operad of a fibrant-cofibrant replacement of the $B$-module $M\otimes_{A}^{\mathbb L}B$, so it does classify derived $\mathcal O$-algebra structures on derived extensions of scalars of $M$.

A consequence of this fact is that $\algebrau{\mathcal{O}}{F}$ is geometric provided $F$ is, e.g.~$F=\ul{\operatorname{Vect}}_{n}$ in many HAG contexts, and $F=\ul{\operatorname{Perf}}$ in the weak complicial and brave new algebraic geometry contexts. Notice that  $\algebrau{\mathtt{Ass}}{\ul{\operatorname{Vect}}_{n}}$ and $\algebrau{\mathtt{uAss}}{\ul{\operatorname{Vect}}_{n}}$ are the immediate generalizations of $\ul{\operatorname{Ass}}_{n}$ and $\ul{\operatorname{uAss}}_{n}$ above.

We finally tackle our goal. The stack  $\algebrau{\mathcal O}{\C V}$ is a contravariant functor in $\mathcal O$. The operad morphism $\phi\colon\mathtt{Ass}\r \mathtt{uAss}$ which models the forgetful functor from unital associative algebras to associative algebras induces a morphism of stacks
\begin{equation}\label{laclau0}
\algebrau{\phi}{\C V}\colon \algebrau{\mathtt{uAss}}{\C V}\To \algebrau{\mathtt{Ass}}{\C V},
\end{equation}
which generalizes $\ul f$ above.

Similarly, $\phi$ induces a morphism of affine stacks
\begin{equation}\label{laclau0.5}
\rmapu_{\operad{\C V}}(\mathtt{uAss},\mathtt{End}_{\Mod{A}}(M))\To \rmapu_{\operad{\C V}}(\mathtt{Ass},\mathtt{End}_{\Mod{A}}(M)).
\end{equation}
which generalizes $f$.

Our main results on the properties of these morphisms are summarized in the following theorem. We refer the reader to \cite{hagII} for the definition of the geometric terms in the statement.

\begin{thm}\label{may}
Consider the morphism $\algebrau{\phi}{F}\colon \algebrau{\mathtt{uAss}}{F}\r \algebrau{\mathtt{Ass}}{F}$ obtained by restricting \eqref{laclau0}. Let $M$ be a perfect $A$-module.
\begin{enumerate}
\item The morphism $\algebrau{\phi}{F}$ is affine, i.e.~$(-1)$-representable. The morphism \eqref{laclau0.5} is obviously affine since it has affine source and target.
\item If $\C V$ is simplicial or complicial, then \eqref{laclau0} and \eqref{laclau0.5} are monomorphisms of stacks, in particular so is $\algebrau{\phi}{F}$.
\item If in addition  $\C V$ is  locally finitely presentable as a category, finitely generated as a model category, and the tensor unit is  finitely presented, then \eqref{laclau0.5}  is a finitely presented morphism of affine stacks and $\algebrau{\phi}{F}$ is categorically locally finitely presented.
\end{enumerate}
\end{thm}

The crucial tool in the proof of (1) is a non-symmetric generalization of one of the main theorems in Rezk's thesis \cite[Theorem 1.1.5]{rezkphd}, which computes the homotopy fibers of the forgetful functor sending an $\mathcal O$-algebra to the underlying object, see Theorem \ref{rezk} below. The proof of (2) is based in the main theorem of \cite{udga}. For the proof of (3) we use Lurie's \cite[Theorem 5.2.3.5]{lurieha}.

The only property of affine Zariski open immersions we do not address in Theorem \ref{may} is flatness. Flatness is defined in arbitrary HAG contexts for morphisms between affine stacks, but otherwise it is only defined for morphisms between geometric stacks whenever flatness is a local property for $\tau$, so it may well happen that it does not make sense to speak about the flatness of $\algebrau{\phi}{F}$ in a given HAG context. This contrasts with the properties considered in Theorem \ref{may}, which are defined in all HAG contexts. Nevertheless, in most HAG contexts of interest $\C V$ is stable, with the remarkable exception of the derived algebraic geometry context. In stable contexts, all morphisms of affine stacks are flat, hence flatness is a local property and all morphisms between geometric stacks are flat, so \eqref{laclau0.5} and $\algebrau{\phi}{F}$ would be flat in this situation if $F$ is geometric. It would be interesting to check whether this remains true in derived algebraic geometry.

To finish, in contrast with the classical situation, we show with an example that the stack $\rmapu_{\operad{\C V}}(\mathtt{Ass},\mathtt{End}_{\Mod{A}}(M))$ need not be finitely presented. Our example is in complicial algebraic geometry, in the two HAG contexts considered in \cite[\S 2.3]{hagII}. It is simply given by $A=\Bbbk=\mathbb Q$ and $M=\Sigma^{n}\mathbb Q$ for $n\leq -2$.

This paper is very much indebted to  ideas and results in \cite{rezkphd} and \cite{hagII}. 

In the derived algebraic geometry context, where $\C V=\Mod{\Bbbk}^{\Delta^{\op}}$ is the category of simplicial modules over a commutative ring $\Bbbk$, the stack $\algebrau{\mathcal{O}}{\ul{\operatorname{Vect}}_{n}}$ was considered in \cite[\S 2.2.6.2]{hagII} for $\mathcal O$ an operad  of projective $\Bbbk$-modules regarded as a constant simplicial operad. However, the homotopy invariance and descent properties were left to the reader. The same happened with $\algebrau{\mathtt{Ass}}{\ul{\operatorname{Perf}}}$ in the complicial algebraic geometry contex, where $\C V=\operatorname{Ch}(\Bbbk)$ is the category of chain complexes over a commutative $\mathbb Q$-algebra $\Bbbk$, see \cite[\S 2.3.3.2 and \S 2.3.5.3]{hagII}. The proofs carried out in this paper subsume these tasks  To\"en and Vezzosi left to the reader.

In the derived algebraic geometry context, the map $\xi^{\mathcal{O}}_{\ul{\operatorname{Vect}}_{n}}\colon \algebrau{\mathcal{O}}{\ul{\operatorname{Vect}}_{n}}\r \ul{\operatorname{Vect}}_{n}$ was considered in \cite[\S 2.2.6.2]{hagII}, for $\mathcal O$ as above. There, it is proved that $\xi^{\mathcal O}_{\ul{\operatorname{Vect}}_{n}}$ is affine and hence $\algebrau{\mathcal{O}}{\ul{\operatorname{Vect}}_{n}}$ is geometric. The main tool used by To\"en and Vezzosi is Rezk's  aforementioned theorem, which is only for (symmetric) operads in $\operatorname{Set}^{\Delta^{\op}}$ or $\Mod{\Bbbk}^{\Delta^{\op}}$. 

Similarly, in the complicial algebraic geometry context, To\"en and Vezzosi consider the map $\xi^{\mathcal{O}}_{\ul{\operatorname{Perf}}}\colon \algebrau{\mathcal{O}}{\ul{\operatorname{Perf}}}\r \ul{\operatorname{Perf}}$ in \cite[\S 2.3.3.2 and \S 2.3.5.3]{hagII}. They invoke Rezk's theorem to show that $\xi^{\mathcal{O}}_{\ul{\operatorname{Perf}}}$ is affine and that $\algebrau{\mathcal{O}}{\ul{\operatorname{Perf}}}$ is geometric. However, Rezk did not prove his theorem for $\operatorname{Ch}(\Bbbk)$, the base symmetric monoidal model category of the complicial algebraic geometry context. Certainly, Rezk's proof does not extend to $\operatorname{Ch}(\Bbbk)$. For instance, Rezk uses in a crucial way the simplicial structure of $\operatorname{Set}^{\Delta^{\op}}$ and $\Mod{\Bbbk}^{\Delta^{\op}}$, but $\operatorname{Ch}(\Bbbk)$  is the paradigm of model category which is not simplicial. Our proof of the generalization of Rezk's theorem follows Rezk in spirit, but it is quite different in practice. In particular, our results fill this gap in \cite[\S 2.3.3.2 and \S 2.3.5.3]{hagII}.

The paper is structured as follows. In Section \ref{oa} we fix terminology concerning model categories and monoidal structures, and recall some notions about operads and their algebras. Section \ref{endinv} considers homotopy invariance properties of endomorphism operads, which play a role in the proof of the generalization of Rezk's theorem, in Section \ref{mapping}. 

The moduli stacks of algebras and algebra structures over an admissible operad are constructed in Section \ref{geometry}. The existence of $\algebrau{\mathcal O}{\C V}$ is far from obvious. The proofs of invariance and descent properties depend heavily on the homotopy theory of non-symmetric operads developed in \cite{htnso,htnso2}. Neither is obvious the existence of $\xi^{\mathcal O}$, nor the contravariant functoriality of $\algebrau{\mathcal O}{\C V}$. We also prove in this section those properties of $\xi^{\mathcal O}$ and the moduli stacks of algebras and algebra structures which hold for any admissible operad $\mathcal O$. 


Finally, in Section \ref{final} we concentrate in $\mathcal O=\mathtt{Ass},\mathtt{uAss}$ and in the maps \eqref{laclau0} and  \eqref{laclau0.5}. Theorem \ref{may} follows directly from the results in that section. 

At the end of the paper, there is a short appendix on a technical property of the tensor unit of a monoidal model category, which is always satisfied by the underlying symmetric monoidal model category $\C V$ of a HAG context. It is also satisfied when the tensor unit is cofibrant. The interest of this axiom comes from the fact that in some HAG contexts of interest the tensor unit is not cofibrant, e.g.~in brave new algebraic geometry.

\subsection*{Acknowledgements}
The author was partially supported
by the Andalusian Ministry of Economy, Innovation and Science under the grant FQM-5713, by the Spanish Ministry of Education and
Science under the MEC-FEDER grant  MTM2010-15831, and by the Government of Catalonia under the grant SGR-119-2009. He is grateful to Behrang Noohi and Fank Neumann for their interest and for conversations related to the contents of this paper.

\section{Operads and  algebras}\label{oa}

This section contains some background about operads and their algebras. 
All operads considered in this paper are non-symmetric. 

We will be mostly dealing with (symmetric) monoidal model categories. Model structures are not really relevant in this section, but we prefer to fix now the standard assumptions needed in most of the paper.

\begin{defn}\label{nota1}
A \emph{monoidal model category} $\C{C}$ in the sense of \cite[Definition 4.2.6]{hmc} is a biclosed   monoidal category  endowed with a model structure such that the \emph{push-out product axiom} \cite[Definitions 3.1]{ammmc} and the \emph{unit axiom} hold. These two axioms imply that the homotopy category $\ho\C C$ has an induced biclosed   monoidal structure \cite[\S4.3]{hmc}. The tensor product is denoted by $\otimes$ and the tensor unit by $\unit$. If we need to distinguish between different monoidal categories we add a subscript, e.g.~$\otimes_{\C C}$ and $\unit_\C{C}$.

We further assume that monoidal model categories satisfy the \emph{monoid axiom} \cite[Definition 9.1]{htnso} and the \emph{strong unit axiom} \cite[Definition A.9]{htnso2}, recalled in Remark \ref{sua} below. They guarantee the existence of transferred model structures on monoids, operads, algebras over operads, etc. See \cite{ammmc,htnso}. The strong unit axiom
allows the transfer of Quillen equivalences and always holds if the tensor unit is cofibrant, see \cite{htnso2}.

We will also assume that all monoidal model categories are cofibrantly generated by sets of generating (trivial) cofibrations with presentable sources.

A \emph{symmetric monoidal model category} $\C{V}$ is a monoidal model category as above whose underlying monoidal category is symmetric. In this case the monoid axiom simplifies, see \cite[Definitions 3.3]{ammmc}.
\end{defn}

For the rest of this section, we fix a symmetric monoidal model category $\C V$.

\begin{defn}\label{defop}
 An \emph{operad} $\mathcal O$ in $\C V$ is a sequence $\mathcal O=\{\mathcal O(n)\}_{n\geq 0}$
of objects in $\C V$ equipped with an \emph{identity},
$$\id{\mathcal{O}}\colon\unit\r \mathcal{O}(1),$$ 
and  \emph{composition laws}, $1\leq i\leq p$, $q\geq 0$,
$$\circ_i\colon \mathcal{O}(p)\otimes \mathcal{O}(q)\To \mathcal{O}(p+q-1),$$
satisfying certain associativity and unit equations, see \cite[Remark 2.6]{htnso}. We refer to $\mathcal O(n)$ as the \emph{arity} $n$ component of $\mathcal O$.

A \emph{morphism of operads} $f\colon \mathcal O\r\mathcal P$ is a sequence of morphisms $f(n)\colon \mathcal O(n)\r\mathcal P(n)$ in $\C V$, $n\geq 0$, compatible with the identities and composition laws in the obvious way. 
We denote $\operad{\C V}$ the model category of operads in $\C V$. An operad morphism $f$ is a fibration (resp.~weak equivalence) if $f(n)$ is a fibration (resp.~weak equivalence) in $\C V$ for all $n\geq 0$ \cite[Theorem 1.1]{htnso}.
\end{defn}

\begin{exm}\label{uass0}
The \emph{unital associative operad} $\mathtt{uAss}$ is defined as follows: $\mathtt{uAss}(n)=\unit$ is the tensor unit for all $n\geq 0$, $\circ_{i}\colon \mathtt{uAss}(p)\otimes \mathtt{uAss}(q)\r \mathtt{uAss}(p+q-1)$ is the unit isomorphism $\unit\otimes\unit\cong \unit$ in all cases, and $\id{\texttt{uAss}}$ is the identity map in $\unit$.

The \emph{associative operad} $\mathtt{Ass}$ is defined by  $\mathtt{uAss}(0)=\varnothing$ the initial object and by the existence of an operad morphism $\phi\colon \mathtt{Ass}\r \mathtt{uAss}$ such that $\phi(n)$ is the identity for $n\geq 1$.
\end{exm}

\begin{rem}\label{examples0}
If $\C{V}=\operatorname{Set}$ is the category of sets, the identity is simply an element $\id{\mathcal{O}}\in\mathcal{O}(1)$ and the associativity and unit equations are:
\begin{enumerate}
\item $(x\circ_i y)\circ_j z=(x\circ_j z)\circ_{i+q-1}y$ if $1\leq j<i$ and $z\in\mathcal{O}(q)$.
\item $(x\circ_i y)\circ_j z=x\circ_i( y\circ_{j-i+1}z)$ if $y\in\mathcal{O}(p)$ and $i\leq j <p+i$.
\item $\id{\mathcal{O}}\circ_1 x=x$.
\item $x\circ_i \id{\mathcal{O}}=x$.
\end{enumerate}
The same happens if $\C{V}=\Mod{\Bbbk}$ is the category of modules over a commutative ring $\Bbbk$.

If $\C{V}=\Mod{\Bbbk}^{\mathbb Z}$ is the category of $\mathbb{Z}$-graded $\Bbbk$-modules then the identity must be in degree $0$, $\id{\mathcal{O}}\in\mathcal{O}(1)_{0}$, and  (1) must be replaced with
\begin{itemize}
\item[$(1')$] $(x\circ_i y)\circ_j z=(-1)^{|y||z|}(x\circ_j z)\circ_{i+q-1}y$ if $1\leq j<i$ and $z\in\mathcal{O}(q)$.
\end{itemize}
This reflects the use of the Koszul sign rule in the definition of the symmetry constraint for the tensor product in $\Mod{\Bbbk}^{\mathbb Z}$.

Furthermore, if $\C{V}=\operatorname{Ch}(\Bbbk)$ is the category of differential graded $\Bbbk$-modules, in addition the identity must be a cycle, $d(\id{\mathcal{O}})=0$, and the differential must behave as a derivation with respect to all composition laws,
$$d(x\circ_{i}y)=d(x)\circ_{i} y+(-1)^{|x|}x\circ_{i}d(y).$$
In this paper, differentials have degree $|d|=-1$, i.e.~we consider chain complexes.

In all these cases, the compatibility conditions an operad morphism $f\colon\mathcal O\r\mathcal P$ must satisfy are
$$f(p)(x)\circ_if(q)(y)=f(p+q-1)(x\circ_iy),\qquad f(1)(\id{\mathcal O})=\id{\mathcal P}.$$
\end{rem}

\begin{defn}\label{nota2}
A \emph{model $\C V$-algebra} $\C C$ is a monoidal model category, in the sense of Definition \ref{nota1}, equipped with a left Quillen functor $z\colon\C{V}\r\C{C}$ and natural isomorphisms
\begin{align*}
\text{\emph{multiplication}}\colon z(X)\otimes_{\C C} z(X')&\st{}\To z(X\otimes_{\C V} X'),\\
\text{\emph{unit}}\colon \unit_{\C C}&\st{}\To z(\unit_{\C V}),\\
\zeta(X,Y)\colon z(X)\otimes_{\C C} Y &\st{}\To Y\otimes_{\C C} z(X),
\end{align*}
satisfying some coherence laws, see \cite[Definition 6.4.1]{borceux2} and \cite[\S 7]{htnso}.

We also assume that $z$ satisfies the \emph{$\mathbb I$-cofibrant axiom} \cite[Definition B.6]{htnso2}. Hence, Quillen equivalent $\C V$-algebras have Quillen equivalent categories of algebras over a same nice enough operad, see \cite[Theorem D.11]{htnso2}. 

A \emph{symmetric model $\C V$-algebra} $\C C$ is a symmetric monoidal model category equipped with a strong symmetric monoidal left Quillen functor $z\colon\C{V}\r\C{C}$ satisfying the $\mathbb I$-cofibrant axiom. 

In Hovey's terminology, a model $\C V$-algebra is the same as a central monoidal $\C V$-model category, and a symmetric model $\C V$-algebra is a symmetric $\C V$-model category, see the paragraph after \cite[Definition 4.2.20]{hmc}.
\end{defn}

For the rest if this section, let us fix a model $\C V$-algebra $\C C$.

\begin{exm}\label{ejemplotrivial}
The trivial example of symmetric model $\C V$-algebra is $\C C=\C V$ and $z$ the identity functor.
\end{exm}

\begin{rem}
The functor $z(-)\otimes Y\colon\C{V}\r \C{C}$
has a right adjoint 
\begin{align*}
\hom_{\C C}(Y,-)\colon&\C{C}\To \C{V}
\end{align*}
for any $Y$ in $\C C$. These morphism objects define a tensored $\C V$-enrichment of $\C C$ over $\C V$, see \cite[Appendix]{anamc}. Moreover, this enrichment can be enhanced to a monoidal $\C V$-category structure on $\C C$. For this, given two objects $Y$ and  $Y'$ in $\C C$, we define the \emph{evaluation morphism}
$$\text{evaluation}\colon z(\hom_{\C C}(Y,Y'))\otimes Y\To Y',$$
as the adjoint of the identity in $\hom_{\C C}(Y,Y')$, and given two other objects $X$ and $X'$ in $\C C$ we definine the morphism in $\C V$
$$\otimes_{\C C} \colon \hom_{\C{C}}(X,X')\otimes_{\C V} \hom_{\C{C}}(Y,Y')\longrightarrow \hom_{\C{C}}(X\otimes_{\C C} Y, X'\otimes_{\C C} Y')$$
as the adjoint of
$$\xymatrix{
z(\hom_{\C{C}}(X,X')\otimes_{\C V} \hom_{\C{C}}(Y,Y'))\otimes_{\C C} X\otimes_{\C C} Y\ar[d]^\cong_{\text{multiplication}^{-1}\otimes_{\C C} \id{X\otimes_{\C C}Y}}\\ 
z(\hom_{\C{C}}(X,X'))\otimes_{\C C} z(\hom_{\C{C}}(Y,Y'))\otimes_{\C C} X\otimes_{\C C} Y
\ar[d]_{\zeta(\hom_{\C{C}}(Y,Y'), X)}^\cong\\ 
z(\hom_{\C{C}}(X,X'))\otimes_{\C C} X\otimes_{\C C} z(\hom_{\C{C}}(Y,Y'))\otimes_{\C C} Y\ar[d]_{\text{evaluation} \otimes_{\C C} \text{evaluation}}\\X' \otimes_{\C C} Y'}
$$
\end{rem}

\begin{defn}
Let $\mathcal{O}$ be an operad in $\C V$. An \emph{$\mathcal{O}$-algebra $A$ in $\C C$} is an object of~$\C C$ equipped with structure morphisms
\begin{equation*}
\nu_n\colon z(\mathcal{O}(n))\otimes A^{\otimes n}\To A,\qquad n\geq0,
\end{equation*}
satisfying compatibility relations with the composition laws and the unit of~$\mathcal{O}$, see \cite[Definition 7.1]{htnso}.

An \emph{$\mathcal{O}$-algebra morphism} $g\colon A\r B$ is a morphism in $\C C$ compatible with the structure morphisms in the obvious way, see again \cite[Definition 7.1]{htnso}.  We denote $\algebra{\mathcal O}{\C C}$ the model category of $\mathcal{O}$-algebras $A$ in $\C C$, where a morphism is a fibration (resp.~weak equivalence) if its underlying morphism in $\C C$ is a fibration (resp.~weak equivalence) \cite[Theorem 1.2]{htnso}.
\end{defn}

\begin{exm}\label{uass1}
Algebras over the unital associative operad 
$\mathtt{uAss}$ in Example \ref{uass0} are unital  associative algebras in $\C C$, i.e.~monoids, and algebras over the associative operad $\mathtt{Ass}$ are associative algebras in $\C C$, i.e.~non-unital monoids.
\end{exm}

\begin{rem}\label{examples2}
Let $\C V=\C C$ and $z$ the identity functor. If $\C V=\operatorname{Set}$ or $\Mod{\Bbbk}$, as in Remark \ref{examples0}, and we denote $$\nu_n(x, a_1,\dots, a_n)=x(a_1,\dots,a_n),$$ then the relations an $\mathcal O$-algebra must satisfy are:
\begin{enumerate}
\item $(x\circ_i y)(a_1,\dots,a_{p+q-1})=x(a_1,\dots,a_{i-1},y(a_i,\dots,a_{i+q-1}),a_{i+q},\dots,a_{p+q-1})$ if $x\in\mathcal{O}(p)$ and $y\in\mathcal{O}(q)$.
 \item $\id{\mathcal O}(a)=a$.
\end{enumerate}

If $\C{V}=\Mod{\Bbbk}^{\mathbb Z}$ then (1) must be replaced with
\begin{itemize}
\item[$(1')$] $(x\circ_i y)(a_1,\dots,a_{p+q-1})$\newline$=(-1)^{|y|\sum_{j=1}^{i-1}|a_j|}x(a_1,\dots,a_{i-1},y(a_i,\dots,a_{i+q-1}),a_{i+q},\dots,a_{p+q-1})$\newline if $x\in\mathcal{O}(p)$ and $y\in\mathcal{O}(q)$.
\end{itemize}
Moreover, if $\C{V}=\operatorname{Ch}(\Bbbk)$, then in addition the following derivation-like formula holds,
\begin{align*}
d(x(a_1,\dots,a_n))&=d(x)(a_1,\dots,a_n)+\sum_{i=1}^{n}(-1)^{|x|+\sum_{j=1}^{i-1}|a_{j}|}x(a_1,\dots,d(a_{i}),\dots,a_n).
\end{align*}

In all these cases the compatibility conditions that an $\mathcal O$-algebra morphism $g\colon A\r B$ must satisfy are $$g(x(a_1,\dots,a_n))=x(g(a_1),\dots,g(a_n)).$$
\end{rem}

Algebras over an operad can be alternatively described by means of endomorphism operads.

\begin{defn}
The \emph{endomorphism operad} of an object $Y$ in $\C C$ is the operad $\texttt{End}_{\C C}(Y)$ in $\C V$ with 
$$\texttt{End}_{\C C}(Y)(n)=\hom_{\C C}(Y^{\otimes n},Y).$$
The identity of this operad is the $\C V$-enriched identity in $Y$. Composition laws are described in \cite[Definition 7.1]{htnso}. 
\end{defn}

\begin{rem}\label{examples3}
In the first (resp.~last) two examples of Remark \ref{examples2}, composition laws in the endomorphism operad are given by equation $(1)$ (resp.~$(1')$).
\end{rem}

\begin{lem}
For any operad $\mathcal{O}$ in $\C V$ and any object $Y$ in $\C C$, there is a bijection between the morphisms $\mathcal{O}\r \mathtt{End}_{\C C}(Y)$ in $\operad{\C V}$ and the $\mathcal{O}$-algebra structures on~$Y$.
\end{lem}

\begin{proof}
The adjoint of the morphism $\mathcal{O}(n)\rightarrow \mathtt{End}_{\C C}(Y)(n)=\hom_{\C{C}}(Y^{\otimes n},Y)$ is the structure morphism $\nu_n\colon z(\mathcal{O}(n))\otimes Y^{\otimes n}\r Y$.
\end{proof}

Diagrams of algebras over an operad $\mathcal O$ can also be descried as $\mathcal O$-algebras in a category of diagrams.


Given a small category $I$, the category of \emph{$I$-shaped diagrams in $\C C$} is the category $\C{C}^I$ of  functors $I\r \C{C}$ and natural transformations between them. This category inherits  from $\C C$  a model $\C V$-algebra structure. A morphism $Y\to Y'$ between diagrams $Y,Y'\colon I\r\C C$ is a fibration (resp.~weak equivalence) if $Y(i)\to Y'(i)$ is a fibration (resp.~weak equivalence) for all objects $i$ in $I$. The tensor product of $Y$ and $Y'$  is defined as $$(Y\otimes Y')(i)=Y(i)\otimes Y'(i),\qquad i\in I,$$ 
and the tensor unit of $\C{C}^I$ is the constant diagram $\unit_{\C{C}^I}(i)=\unit_{\C{C}}$.  The rest of the structure is given by 
\begin{align*}
 z^I\colon&\C{V}\To Z(\C{C}^I), &z^I(X)(i)&=z(X),&\zeta^I(X,Y)(i)&=\zeta(X,Y(i)).
\end{align*}
The right adjoint of $z^I(-)\otimes Y\colon\C{V}\r\C{C}^I$ is the functor
$$\hom_{\C{C}^I}(Y,-)\colon\C{C}^I\To \C{V}$$
defined by the following end in $\C V$ \cite[IX.5]{cwm2},
$$\hom_{\C{C}^I}(Y,Y')=\int_{i\in I}\hom_{\C{C}}(Y(i),Y'(i)).$$
The next result follows readily from the previous lemma and the universal property of an end.

\begin{cor}\label{tonteria}
For any operad $\mathcal{O}$ in $\C V$, any small category~$I$, and any diagram  $Y\colon I\r\C{C}$, 
there is a bijection between the morphisms $\mathcal{O}\r \mathtt{End}_{\C{C}^I}(Y)$ in $\operad{\C V}$ and the collections of  $\mathcal{O}$-algebra structures on the objects $Y(i)$ in $\C{C}$, $i\in I$, such that the morphisms in the diagram are $\mathcal{O}$-algebra morphisms.
\end{cor}

\begin{defn}
A \emph{$\C V$-algebra functor} is a lax monoidal functor between $\C V$-algebras $F\colon\C C\r\C D$ equipped with a monoidal natural isomorphism $Fz_{\C C}(X)\cong z_{\C D}(X)$ such that the diagram in \cite[Definition 4.1.11]{hmc} commutes (after replacing Hovey's $i$ with our $z$, which is just a matter of notation).
\end{defn}

\begin{rem}
A $\C V$-algebra functor $F\colon\C C\r\C D$ is $\C V$-enriched. For any two objects $Y$ and $Z$ of $\C C$, the morphism $F(Y,Z)\colon \hom_{\C C}(Y,Z)\r\hom_{\C D}(F(Y),F(Z))$ in $\C V$ is the adjoint of
$$\xymatrix{z_{\C D}(\hom_{\C C}(Y,Z))\otimes_{\C D}F(Y)\ar[d]^\cong\\
Fz_{\C C}(\hom_{\C C}(Y,Z))\otimes_{\C D}F(Y)\ar[d]^{\text{mult.}}\\
 F(z_{\C C}(\hom_{\C C}(Y,Z))\otimes_{\C C}Y)\ar[d]^{F(\text{evaluation})}\\
 F(Z).}$$
\end{rem}

A functor $F\colon J\r I$ between small categories induces by precomposition a strong monoidal $\C V$-algebra functor 
\begin{equation*}
F^*\colon \C{C}^I\To \C{C}^J, \qquad F^*(Y)=Y F.
\end{equation*}
The morphisms between morphism objects in $\C V$,
$$F^*(Y,Z)\colon\hom_{\C{C}^I}(Y,Z)\To \hom_{\C{C}^J}(YF,ZF),$$
are defined by the universal property of an end. These morphisms give rise to morphisms in $\operad{\C V}$ between endomorphism operads,
\begin{equation}\label{puba}
F^*\colon \mathtt{End}_{\C{C}^I}(Y)\To \mathtt{End}_{\C{C}^J}(Y F).
\end{equation}

Denote $\Delta$ the \emph{simplex category}, whose objects are the finite ordinals $$\deltan{n}=\{0<\cdots<n\}, \qquad n\geq 0,$$ and morphisms are non-decreasing maps. These ordinals are regarded as categories with morphisms going upwards $i\r j$, $i\leq j$. The category $\Delta$ is generated by the \emph{coface} and \emph{codegeneracy} maps,
$$d^{i}\colon \deltan{n-1}\To \deltan{n},\qquad s^{i}\colon \deltan{n+1}\To \deltan{n},\qquad 0\leq i\leq n,$$
which satisfy the duals of the usual simplicial relations. As usual, we denote $(d^{i})^{*}=d_{i}$ and $(s^{i})^{*}=s_{i}$.

A morphism in $\C C$ is the same as a functor $\deltan{1}\r\C C$. Applying the previous corollary to $I=\deltan{1}$, we deduce  the following characterization of $\mathcal{O}$-algebra morphisms.

\begin{cor}\label{morfismos}
Given an operad $\mathcal{O}$ in $\C V$ and two $\mathcal{O}$-algebras in $\C C$, $X$ and $Y$, defined by morphisms $f_X\colon\mathcal{O}\r \mathtt{End}_{\C C}(X)$ and $f_Y\colon\mathcal{O}\r \mathtt{End}_{\C C}(Y)$ in $\operad{\C V}$, there is a bijection between the morphisms $g\colon X\r Y$ in $\algebra{\mathcal{O}}{\C C}$ 
and the morphisms $g\colon X\r Y$ in ${\C C}$ such that there exists a morphism $f_g\colon\mathcal{O}\r \mathtt{End}_{\C{C}^{\deltan{1}}}(g)$ making the following diagram commutative 
$$\xymatrix{&\mathcal{O}
\ar[ld]_-{f_X}
\ar[d]^-{f_g}
\ar[rd]^-{f_Y}&\\
\mathtt{End}_{\C C}(X)&\mathtt{End}_{\C{C}^{\deltan{1}}}(g)\ar[l]_-{d_{1}}\ar[r]^-{d_{0}}&\mathtt{End}_{\C C}(Y).}$$
\end{cor}

In this corollary, the morphism $f_g$ is unique provided it exists since, by definition of end, we have pull-back diagrams as follows, $n\geq 0$,
\begin{equation}\label{pb}
\xy 
(-12,0)*{\mathtt{End}_{\C{C}}(X)(n)=\hom_{\C C}(X^{\otimes n},X)},
(40,0)*{\mathtt{End}_{\C{C}^{\deltan{1}}}(g)(n)},
(0,-15)*{\hom_{\C C}(X^{\otimes n},Y)},
(54,-15)*{\hom_{\C C}(Y^{\otimes n},Y)=\mathtt{End}_{\C{C}}(Y)(n)},
(20,-8)*{\scriptstyle\text{pull}},
\ar@{<-}@/^12pt/(-23,3);(35,3)^{d_1}
\ar@{<-}(13,0);(29,0)
\ar(0,-3);(0,-12)_{\hom_{\C C}(X^{\otimes n},g)}
\ar(40,-3);(40,-12)
\ar@/^5pt/(50,-1);(65,-12)^{d_0}
\ar@{<-}(13,-15);(29,-15)_{\hom_{\C C}(g^{\otimes n},Y)}
\endxy
\end{equation}

\section{Homotopy invariance of endomorphism operads}\label{endinv}

Let $\C V$ be a symmetric monoidal model category and $\C C$ a model $\C V$-algebra. In this section we establish the homotopy invariance properties in $\operad{\C V}$ of endomorphism operads of objects in $\C C$. This is a standard way of transferring algebra structures along weak equivalences in $\C C$, compare \cite[Theorem 3.5]{ahto}, and it will have further applications in Section \ref{mapping}.

\begin{lem}\label{ultimo}
Given a cofibrant object $Y$  in $\C C$, the adjoint pair 
$$\xymatrix@C=50pt{{\C V}\ar@<.5ex>[r]^-{z(-)\otimes Y}& {\C C}\ar@<.5ex>[l]^-{\hom_{\C C}(Y,-)}}$$
is a Quillen pair.
\end{lem}

\begin{proof}
Since $z$ is a left Quillen functor and $Y$ is cofibrant then $z(-)\otimes Y$ is a left Quillen functor by the push-out product axiom.
\end{proof}

\begin{cor}\label{esfibrante}
If $Y$ is a fibrant-cofibrant object in $\C C$ then $\mathtt{End}_{\C C}(Y)$ is a fibrant operad in $\operad{\C V}$.
\end{cor}

\begin{proof}
The tensor powers $Y^{\otimes n}$ are cofibrant for $n\geq 1$ by the push-out product axiom. Hence, by the previous lemma, the objects $\mathtt{End}_{\C C}(Y)(n)=\hom_{\C C}(Y^{\otimes n},Y)$ are fibrant in $\C V$ since $Y$ is fibrant and $\hom_{\C C}(Y^{\otimes n},-)$ is a right Quillen functor. Similarly, $\mathtt{End}_{\C C}(Y)(0)=\hom_{\C C}(\unit,Y)$ is fibrant since $Y$ is fibrant and $\hom_{\C C}(\unit,-)$ is a right Quillen functor. It is actually the right adjoint of the left Quillen functor~$z$.
\end{proof}

\begin{lem}\label{comoma}
If $f\colon X\r Y$ is a cofibration (resp.~trivial cofibration) 
and $Z$ is a fibrant object in $\C C$, then the induced morphism
$$\hom_{\C C}(f,Z)\colon \hom_{\C C}(Y,Z)\To \hom_{\C C}(X,Z)$$
is a fibration (resp.~trivial fibration).
\end{lem}

\begin{proof}
Consider a commutative square of solid arrows in $\C{V}$
\begin{equation}\label{acomoma}
\xymatrix{
U\ar[d]_g\ar[r]^-h&\ar[d]^{\hom_{\C C}(f,Z)} \hom_{\C C}(Y,Z)\\
V\ar[r]_-k\ar@{-->}[ru]^-{l}&\hom_{\C C}(X,Z)
}
\end{equation}
where $g$ is a trivial cofibration (resp.~cofibration). 
We must construct a diagonal morphism $l$ 
such that the two triangles commute.

The solid diagram \eqref{acomoma} is the same as a commutative square in $\C{C}$
$$\xymatrix@C=30pt{
z(U)\otimes X\ar[d]_{z(g)\otimes X}\ar[r]^-{z(U)\otimes f}&\ar[d]^{\bar{h}} z(U)\otimes Y\\
z(V)\otimes X\ar[r]_-{\bar k}&Z}$$
where $\bar h$ and $\bar k$ are the adjoints of $h$ and $k$, respectively. 
We will consider the induced morphism from the push-out of the left upper corner to $Z$,
$$\xymatrix@C=30pt{
z(U)\otimes X\ar[d]_{z(g)\otimes X}\ar[r]^-{z(U)\otimes f}\ar@{}[rd]|{\text{push}}&\ar@/^15pt/[ddr]^{\bar{h}} z(U)\otimes Y\ar[d]^-{\bar g}\\
z(V)\otimes X\ar@/_15pt/[rrd]_-{\bar k}\ar[r]_-{\bar f}&W\ar[rd]^-{r}\\
&&Z}$$
Since $z$ is a left Quillen functor, $z(g)$ is a trivial cofibration (resp.~cofibration), hence the push-out product $z(g)\odot f\colon W \r z(V)\otimes Y$ of $z(g)$ and $f$, see \cite[\S4]{htnso}, is a trivial cofibration by the push-out product axiom. Since $Z$ is fibrant, there exists a morphism $\bar{l}\colon z(V)\otimes Y\r Z$ fitting into a commutative diagram
$$\xymatrix{W\ar@{ >->}[d]_{z(g)\odot f}^\sim\ar[r]^-{r}&Z\ar@{->>}[d]\\z(V)\otimes Y\ar[ru]_-{\bar l}\ar[r]&{*}}$$
Here $*$ denotes the final object.
We can take $l\colon V\r \hom_{\C C}(Y,Z)$ to be the adjoint of $\bar l$. 
\end{proof}

The following result is a consequence of the previous lemma and Ken Brown's lemma.

\begin{cor}\label{belcor}
For any fibrant object $Z$ in $\C C$, the functor $\hom_{\C C}(-,Z)\colon\C{C}^{\op}\r\C{V}$ takes weak equivalences between cofibrant objects in $\C C$ to weak equivalences in $\C V$.
\end{cor}

An immediate consequence of the following proposition is that the endomorphism operad of a fibrant-cofibrant object in $\C C$ is an invariant of its weak homotopy type.

\begin{prop}\label{uno}
Let $f\colon X\r Y$ be a morphism in $\C C$. Consider the induced morphisms between endomorphism operads in~$\operad{\C V}$,
$$\mathtt{End}_{\C{C}}(X)\st{d_{1}}\longleftarrow\mathtt{End}_{\C{C}^{\deltan{1}}}(f)\st{d_{0}}\To \mathtt{End}_{\C{C}}(Y).$$
\begin{enumerate}
 \item If $f$ is a (trivial) fibration and $X$ is cofibrant then $d_0$ is a (trivial) fibration.
  \item If $f$ is a trivial fibration  between fibrant-cofibrant objects then $d_1$ is a weak equivalence.
 \item If $f$ is a (trivial) cofibration  between cofibrant objects and $Y$ is fibrant then $d_1$ is a (trivial) fibration.
 
  \item If $f$ is a trivial cofibration  between fibrant-cofibrant objects  then $d_0$ is a weak equivalence.
\end{enumerate}
\end{prop}

\begin{proof}
Consider the pull-back diagram \eqref{pb}. 
By the push-out product axiom, the tensor powers $X^{\otimes n}$ are cofibrant in the four cases. 
Hence, under the assumptions of (1) and (2),  $\hom_{\C C}(X^{\otimes n},f)$ is a (trivial) fibration by Lemma \ref{ultimo}. Now $(1)$ follows from the fact that  (trivial) fibrations  are closed under pull-backs. 

Under the hypotheses of $(2)$, the push-out product axiom and Ken Brown's lemma show that the tensor powers $f^{\otimes n}$ are weak equivalences between cofibrant objects hence $\hom_{\C C}(f^{\otimes n},Y)$ is a weak equivalence between fibrant objects by Lemma \ref{comoma} and Corollary \ref{belcor}. Therefore $(2)$ follows, since the pull-back of a weak equivalence between fibrant objects along a fibration is a weak equivalence.

Under the assumptions of $(3)$ and $(4)$, the tensor powers $f^{\otimes n}$ are (trivial) cofibrations between cofibrant objects by the push-out product axiom. Hence, $\hom_{\C C}(f^{\otimes n},Y)$ is a (trivial) fibration between fibrant objects by Lemma \ref{comoma}. Notice that (3) follows from the same reason as (1). Moreover, under the hypotheses of (4), $\hom_{\C C}(X^{\otimes n},f)$ is a weak equivalence between fibrant objects by Lemma \ref{ultimo} and Ken Brown's lemma, hence (4) follows from the same reason as (2).
\end{proof}

%
%

The following result follows from the usual lifting properties, Proposition \ref{uno} and Corollary \ref{morfismos}.

\begin{cor}
Let $\mathcal O$ be a cofibrant operad in $\C V$ and $f\colon X\st{\sim}\r Y$ a weak equivalence in $\C C$.
\begin{enumerate}
\item If $f$ is a trivial fibration, $X$ is cofibrant, and $Y$ is an $\mathcal O$-algebra, then there exists an $\mathcal O$-algebra structure on $X$ such that $f$ becomes a morphism of $\mathcal O$-algebras.

\item If $f$ is a trivial cofibration, $X$ is cofibrant, $Y$ is fibrant, and $X$ is an $\mathcal O$-algebra, then there exists an $\mathcal O$-algebra structure on $Y$ such that $f$ becomes a morphism of $\mathcal O$-algebras.
\end{enumerate}
\end{cor}

The following corollary of Proposition \ref{uno} will be very useful in our study of spaces of algebras.

\begin{cor}\label{hitech}
Let $f_{i}\colon X_{i}\r X_{i+1}$ be trivial fibrations between fibrant-cofibrant objects in $\C C$, $0\leq i\leq n$. The induced morphisms between endomorphism operads in~$\operad{\C V}$,
$$\mathtt{End}_{\C{C}^{\deltan{n}}}(X_{0}\r\cdots \r
X_{n})
\st{d_{n+1}}\longleftarrow
\mathtt{End}_{\C{C}^{\deltan{n+1}}}(X_{0}\r\cdots \r
X_{n}\r X_{n+1})\st{d_0^{n+1}}\To \mathtt{End}_{\C{C}}(X_{n+1})$$
are a weak equivalence and a trivial fibration, respectively. Moreover, these three operads are fibrant in $\operad{\C V}$.
\end{cor}

\begin{proof}
By induction on $n$. The case $n=0$ follows directly from Proposition \ref{uno} and Corollary \ref{esfibrante}. Assume $n>0$. 
Consider the following pull back diagram,
$$\xymatrix@C=50pt{
\mathtt{End}_{\C{C}^{\deltan{n}}}(X_{0}\r\cdots \r
X_{n})(m)\ar[d]_{d_0^n}
&\mathtt{End}_{\C{C}^{\deltan{n+1}}}(X_{0}\r\cdots \r
X_{n}\r X_{n+1})(m)
\ar@<8ex>[ddd]^{d_0^{n+1}}
\ar@{}[lddd]|{\text{pull}}\ar[l]_-{d_{n+1}}\\
\mathtt{End}_{\C{C}}(
X_{n})(m)\ar@{=}[d]\\
\hom_{\C C}(X_{n}^{\otimes m},X_{n})\ar[d]_-{\hom_{\C C}(X_{n}^{\otimes m},f_{n})}\\
\hom_{\C C}(X_{n}^{\otimes m},X_{n+1})
&\hom_{\C C}(X_{n+1}^{\otimes m},X_{n+1})
=\mathtt{End}_{\C{C}}(X_{n+1})\ar[l]_-{\hom_{\C C}(f_{n}^{\otimes m},X_{n+1})}}$$
The morphism $d_0^n$ is a trivial fibration by induction hypothesis. We showed within the proof of Proposition \ref{uno} that, under the hypotheses of this corollary,   the morphisms 
$\hom_{\C C}(X_{n}^{\otimes m},f_{n})$ and 
$\hom_{\C C}(f_{n}^{\otimes m},X_{n+1})$ are  a trivial fibration  and a weak equivalence between fibrant objects, respectively. Hence, the first part of the statement follows from the facts that trivial fibrations are closed under composition and pull-backs and that weak equivalences between fibrant objects are closed under pull-backs along fibrations. Finally, the first and last operads are fibrant by induction hypothesis, and the middle one is fibrant since $d_{0}^{n+1}$ is a trivial fibration with fibrant target.
\end{proof}

\section{Spaces of algebras}\label{mapping}

The main result of this section (Theorem \ref{rezk}) generalizes Rezk's \cite[Theorem 1.1.5]{rezkphd} in the non-symmetric context.

We will be concerned with the \emph{classifying space} $|w\C M|$ (also called \emph{nerve}) of the subcategory of weak equivalences $w\C M$ in a model category $\C M$, and more generally in appropriate subcategories of a model category. The classifying space $\abs{w\C M}$ is sometimes referred to as the \emph{classification complex} or \emph{space} of $\C M$, compare \cite{acpdss}. This need not be an honest simplicial set because $\C M$ may have a proper class of objects. Moreover, it need not be homotopically small in the sense of \cite{fcha}. We will overlook this fact since there are well known ways of patching these problems: restricting to the closure of a set of objects under weak equivalences, working with Grothendieck universes, etc.

Denote $\C M_c$ and $\C M_f$ the full subcategories of cofibrant and fibrant objects in $\C M$, respectively. The inclusions induce weak equivalences $\abs{w\C M_c}\simeq\abs{w\C M}\simeq \abs{w\C M_f}$ \cite[Lemma 4.2.4]{rezkphd}. Moreover, a Quillen equivalence $F\colon \C M\rightleftarrows\C N\colon G$ induces weak equivalences $\abs{w\C M_c}\simeq \abs{w\C N_c}$ and $\abs{w\C N_f}\simeq \abs{w\C M_f}$ by \cite[Proposition 2.3, Exemple 2.5 and Th\'eor\`eme 2.9]{ikted}. Therefore, if in addition $G$ or $F$ preserves weak equivalences then it induces a weak equivalence $\abs{w\C M}\simeq \abs{w\C N}$.

In this paper, the word space is a synonym of simplicial set. Bisimplicial sets will be regarded as spaces via the diagonal construction.

Let $\C V$ be a symmetric monoidal model category and $\C C$ a model $\C V$-algebra.
A morphism $f\colon\mathcal{O}\r\mathcal{P}$ in $\operad{\C V}$ induces a Quillen pair of change of operad functors,
\begin{equation}\label{qp}
\xymatrix{\algebra{\mathcal{O}}{\C C}\ar@<.5ex>[r]^-{f_{*}}& \algebra{\mathcal{P}}{\C C}.\ar@<.5ex>[l]^-{f^*}}
\end{equation}
The functor $f^*$ restricts the action of $\mathcal{P}$ to $\mathcal{O}$ along $f$ and is the identity on underlying objects in $\C C$, hence it preserves 
fibrations and weak equivalences. 
The functor $f_*$ is left adjoint to $f^*$. This Quillen pair is a Quillen equivalence if $f$ is a weak equivalence and the operads $\mathcal O$ and $\mathcal P$ are admissible in the sense of the following definition, see \cite[Theorem D.4]{htnso2}.

\begin{defn}\label{admissible}
An object $X$ in $\C V$ is \emph{$\unit$-cofibrant} if there exists a cofibration $\unit\into X$ from the tensor unit $\unit$. An operad $\mathcal O$ in $\C V$ is \emph{admissible} if each $\mathcal O(n)$ is cofibrant or $\unit$-cofibrant, $n\geq 0$.
\end{defn}

\begin{rem}
If $\mathcal O$ is a cofibrant operad then $\mathcal O(n)$ is cofibrant for $n\neq 1$ and $\unit$-cofibrant for $n=1$ \cite[Corollary C.3]{htnso2}, hence cofibrant operads are admissible. The unital associative operad in Example \ref{uass0} is $\unit$-cofibrant in all arities, and the associative operad is $\unit$-cofibrant in positive arities and cofibrant in arity $0$, so they are also admissible despite they are not cofibrant.
\end{rem}

%
%

If $\mathcal{O}^\bullet$ is a cosimplicial operad then the contravariant change of base operad functors, like $f^*$ in \eqref{qp}, induced by cofaces and codegeneracies give rise to a simplicial category $\algebra{\mathcal{O}^\bullet}{\C C}$. 

\begin{lem}\label{1we}
If $\mathcal{O}^\bullet$ is a cosimplicial resolution in the sense of \cite[\S 4.3]{fcha} of an admissible operad $\mathcal{O}$ in $\C V$, then there is a weak equivalence
$$\abs{w\algebra{\mathcal{O}}{\C C}}\st{\sim}\To \abs{ w\algebra{\mathcal{O}^\bullet}{\C C}}.$$
\end{lem}

\begin{proof}
The operads $\mathcal O^{n}$ are cofibrant, hence admissible, $n\geq 0$, 
so faces and degeneracies in $\algebra{\mathcal{O}^\bullet}{\C C}$  are right adjoints of a Quillen equivalence, therefore they induce weak equivalences on nerves. This implies that the iterated degeneracies induce a weak equivalence
$$\abs{w\algebra{\mathcal{O}^0}{\C C}}\st{\sim}\To \abs{ w\algebra{\mathcal{O}^\bullet}{\C C}}.$$
For the same reason, the weak equivalence $ \mathcal{O}^0\st{\sim}\r\mathcal{O}$ induces a weak  equivalence 
$$\abs{w\algebra{\mathcal{O}}{\C C}}\st{\sim}\To \abs{w\algebra{\mathcal{O}^0}{\C C}}.$$
\end{proof}

Denote $fw\C{C}_{fc}$ the category of fibrant-confibrant objects in $\C{C}$ and trivial fibrations between them. Under the assumptions of the previous lemma, define the simplicial category $\C{D}_{\bullet}$ as the following pull-back,
\begin{equation}\label{db}
\xymatrix{\C{D}_{\bullet}\ar[r]\ar[d]_{\rho'_\bullet}\ar@{}[rd]|{\text{pull}}&
w\algebra{\mathcal{O}^\bullet}{\C C}\ar[d]^{\rho_\bullet}\\
fw\C{C}_{fc}\ar[r]_-{\text{inclusion}}&\C C}
\end{equation}
Here we regard the two categories in the bottom as constant simplicial categories, and $\rho_\bullet$ is the forgetful simplicial functor. Notice that $\C{D}_n$ is the category of trivial fibrations between $\mathcal{O}^n$-algebras whose underlying objects in $\C C$ are fibrant and cofibrant.

\begin{lem}\label{2we}
In the situation of the previous paragraph, the horizontal (simplicial) functors in \eqref{db} induce weak equivalences
$$\abs{ fw\C{C}_{fc}}\st{\sim}\To \abs{ w\C{C}},\qquad\abs{\C{D}_\bullet}\st{\sim}\To\abs{ w\algebra{\mathcal{O}^\bullet}{\C C}}.$$
\end{lem}

\begin{proof}
The first part of the statement follows from \cite[Lemmas 4.2.4 and 4.2.5]{rezkphd}. Moreover, by the same results together with \cite[Corollaries D.2 and D.3]{htnso2}, the inclusion $\C{D}_n\subset w\algebra{\mathcal{O}^n}{\C C}$ induces a weak equivalence $$\abs{\C{D}_n}\st{\sim}\To \abs{w\algebra{\mathcal{O}^n}{\C C}},\quad n\geq 0,$$ hence the second map in the statement is also a weak equivalence.
\end{proof}

Recall that the \emph{category of simplicies} $\Delta K$ of a simplicial set $K$ is the category whose objects are pairs $(\deltan{n},x)$ with $n\geq 0$ and $x\in K_n$. A morphism $\sigma\colon (\deltan{n},x)\r (\deltan{m},y)$ in $\Delta K$ is a morphism $\sigma\colon\deltan{n}\r\deltan{m}$ in $\Delta$ such that the induced map $\sigma^*\colon K_m\r K_n$ takes $y$ to $x$, $\sigma^*(y)=x$. This category comes equipped with a natural projection functor $p_K\colon \Delta K\r\Delta$, $p_K(\deltan{n},x)=\deltan{n}$. This construction defines a functor from the category of simplicial sets to the category of small categories over $\Delta$,
\begin{align*}
\operatorname{Set}^{\Delta^{\op}}&\To\operatorname{Cat}\downarrow\Delta, \\
K&\;\mapsto \; p_K.
\end{align*}

\begin{lem}\label{natin}
If $Y$ is a fibrant-cofibrant object in~$\C{C}$ and  $\mathcal{O}^\bullet$ is a cosimplicial resolution  in $\operad{\C V}$, then, with the notation in \eqref{db}, 
 $\abs{\rho'_\bullet\downarrow Y}$ is  weakly equivalent to $\operad{\C V}(\mathcal{O}^\bullet,\mathtt{End}_{\C C}(Y))$. 
\end{lem}

\begin{proof}
We will use an alternative construction of the bisimplicial set $\abs{\rho'_\bullet\downarrow Y}$ in terms of endomorphism operads. By Corollary \ref{tonteria}, 
\begin{equation}\label{anatin}
\abs{\rho'_t\downarrow Y}_s=\coprod_{X_0\r\cdots\r X_s\r Y\text{ in }fw\C{C}_{fc}}\operad{\C V}(\mathcal{O}^t, \mathtt{End}_{\C{C}^{\deltan{s}}}(X_0\r\cdots\r X_s)).
\end{equation}
Notice that the set indexing this coproduct is $\abs{(fw\C{C}_{fc})\downarrow Y}_s$. In order to describe the bisimplicial structure of $\abs{\rho'_\bullet\downarrow Y}$  in terms of the right hand side of \eqref{anatin}, we consider the functor
\begin{align*}
E\colon (\Delta\abs{\C{C}})^\op&\To\operad{\C V},\\
(\deltan{n},X_0\r\cdots \r X_n)&\;\mapsto\;\mathtt{End}_{\C{C}^{\deltan{n}}}(X_0\r\cdots \r X_n).
\end{align*}
Notice that $X_0\r\cdots \r X_n$ is a functor $X\colon\deltan{n}\r\C{C}$. Given a morphism $\sigma\colon (\deltan{n},X)\r (\deltan{m},X')$ in the category of simplices, the induced morphism $E(\sigma)$ is $\sigma^*\colon \mathtt{End}_{\C{C}^{\deltan{m}}}(X')\r \mathtt{End}_{\C{C}^{\deltan{n}}}(X)$, here we use the notation in \eqref{puba}. We also consider the functor
\begin{align*}
 F_Y\colon (fw\C{C}_{fc})\downarrow Y&\To\C{C},\\
(X\r Y)&\;\mapsto\; X,
\end{align*}
and the composite functor
$$
\xymatrix@C=20pt{
(\Delta\abs{(fw\C{C}_{fc})\downarrow Y})^\op
\ar[rr]^-{(\Delta\abs{F_Y})^\op}&&
(\Delta\abs{\C{C}})^\op
\ar[r]^-{E}&
\operad{\C V}
\ar[rr]^-{\operad{\C V}(\mathcal{O}^\bullet,-)}&&
\operatorname{Set}^{\Delta^\op}.}
$$
Taking the left Kan extension \cite[X.3]{cwm2} of this functor along the opposite of the natural proyection from the category of simplices to $\Delta$, we obtain a bisimplicial set $$\operatorname{Lan}_{p^{\op}_{\abs{(fw\C{C}_{fc})\downarrow Y}}}\operad{\C V}(\mathcal{O}^\bullet,E(\Delta\abs{ F_Y})^\op).$$
One can easily check that the $(s,t)$ set of this bisimplicial set is the right hand side of \eqref{anatin}. Moreover, this defines an isomorphism between this bisimplicial set and $\abs{\rho'_\bullet\downarrow Y}$.

We need two more functors
\begin{align*}
L_Y,C_Y\colon \Delta\abs{(fw\C{C}_{fc})\downarrow Y}&\To\Delta\abs{\C{C}},\\
L_Y(\deltan{n},X_0\r\cdots \r X_n\r Y)&=(\deltan{n+1},X_0\r\cdots \r X_n\r Y),\\ L_Y(d^i)&=d^i,\quad L_Y(s^i)=s^i,\\
 C_Y(\deltan{n},X_0\r\cdots \r X_n\r Y)&=(\deltan{0},Y),
\end{align*}
where $C_Y$ is constant, 
and two natural transformations,
\begin{align*}
\delta_Y\colon \Delta\abs{F_Y}&\Longrightarrow L_Y,&
\delta_Y(\deltan{n},X_0\r\cdots \r X_n\r Y)&=d^{n+1},\\
\varsigma_Y\colon C_Y&\Longrightarrow L_Y,&
\varsigma_Y(\deltan{n},X_0\r\cdots \r X_n\r Y)&=(d^0)^{n+1}.
\end{align*}
For simplicity, denote $K=\abs{(fw\C{C}_{fc})\downarrow Y}$.
We claim that the following  morphisms of bisimplicial sets are weak equivalences,
$$
\xymatrix{\operatorname{Lan}_{p^{\op}_{K}}\operad{\C V}(\mathcal{O}^\bullet,E(\Delta\abs{ F_Y})^\op)
\\
\operatorname{Lan}_{p^{\op}_{K}}\operad{\C V}(\mathcal{O}^\bullet,EL_Y^\op)
\ar[u]_-{\operatorname{Lan}_{p^{\op}_{K}}\operad{\C V}(\mathcal{O}^\bullet,E(\delta_Y))}
\ar[d]^-{\operatorname{Lan}_{p^{\op}_{K}}\operad{\C V}(\mathcal{O}^\bullet,E(\varsigma_Y))}\\
\operatorname{Lan}_{p^{\op}_{K}}\operad{\C V}(\mathcal{O}^\bullet,E C_Y^\op)}
$$
It is enough to notice that, at the $(n,\bullet)$ simplicial set, we have the coproduct indexed by 
$K_n$ of the morphisms of simplicial sets obtained by applying $\operad{\C V}(\mathcal{O}^\bullet,-)$ to the weak equivalences between fibrant operads in Corollary \ref{hitech}, $n\geq 0$,
$$
\xy
(0,0)*{\coprod \limits_{X_0\r\cdots\r X_n\r Y\text{ in }fw\C{C}_{fc}}\operad{\C V}(\mathcal{O}^\bullet, \mathtt{End}_{\C{C}^{\deltan{n}}}(X_0\r\cdots\r X_n))},
(0,-15)*{\coprod\limits_{X_0\r\cdots\r X_n\r Y\text{ in }fw\C{C}_{fc}}\operad{\C V}(\mathcal{O}^\bullet, \mathtt{End}_{\C{C}^{\deltan{n+1}}}(X_0\r\cdots\r X_n\r Y))},
(0,-30)*{\coprod\limits_{X_0\r\cdots\r X_n\r Y\text{ in }fw\C{C}_{fc}}\operad{\C V}(\mathcal{O}^\bullet, \mathtt{End}_{\C{C}}(Y))}
\ar(10,-10);(10,-1)_-\sim^{(d_{n+1})_*}
\ar(10,-16);(10,-25)^-\sim_{(d_0^{n+1})_*}
\endxy
$$
The functor $\operad{\C V}(\mathcal{O}^\bullet,-)$ takes the weak equivalences between fibrant operads in Corollary \ref{hitech} to weak equivalences of simplicial sets by \cite[Corollary 6.4]{fcha}.

In order to complete the definition of the weak equivalence claimed in the statement we notice that
$$\operatorname{Lan}_{p^{\op}_{K}}\operad{\C V}(\mathcal{O}^\bullet,EC_Y^\op)
=
\abs{(fw\C{C}_{fc})\downarrow Y}\times \operad{\C V}(\mathcal{O}^\bullet, \mathtt{End}_{\C{C}}(Y)).$$
The category $(fw\C{C}_{fc})\downarrow Y$ has a final object, the identity in $Y$, hence its nerve is  contractible and  the projection onto the second factor of the product is a weak equivalence. 
\end{proof}

We can now proceed with the main theorem of this section. 

\begin{thm}\label{rezk}
Let $\mathcal{O}$ be an admissible operad in $\C V$. 
The homotopy fiber of the map 
$$\abs{w\algebra{\mathcal{O}}{\C C}}\To\abs{w\C C}$$
induced by the forgetful functor $\algebra{\mathcal{O}}{\C C}\r\C C$
at a fibrant-cofibrant object $Y$ is 
$\rmap_{\operad{\C V}}(\mathcal{O},\mathtt{End}_{\C C}(Y))$. 
\end{thm}

\begin{proof}
By Lemma \ref{natin} and the homotopy invariance of endomorphism operads of fibrant-cofibrant objects (Proposition \ref{uno}), we can apply  \cite[Lemma 4.2.2]{rezkphd} to the simplicial functor $\rho'_\bullet\colon\C{D}_\bullet\r fw\C{C}_{fc}$ in  \eqref{db}. This proves that the homotopy fiber of $|\rho'_\bullet|$ at $Y$ is $\abs{\rho'_\bullet\downarrow Y}$, since $|fw\C{C}_{fc}\downarrow Y|$ is contractible because $fw\C{C}_{fc}\downarrow Y$ has a final object, the identity map $Y\r Y$. Now Lemmas \ref{1we} and \ref{2we} complete the proof. 
\end{proof}

Mapping spaces take homotopy colimits in the first variable to homotopy limits. Hence, we deduce the following result.

\begin{cor}\label{rezkcor}
The contravariant functor $\mathcal O\mapsto \abs{w\algebra{\mathcal{O}}{\C C}}$ defined by the contravariant change of base operad functors, like $f^*$ in \eqref{qp}, takes homotopy colimits of admissible operads to homotopy limits of spaces.
\end{cor}

\section{Stacks of algebras}\label{geometry}

In this section we place ourselves in a homotopical algebraic geometry (HAG) context $(\C V,\C V_0,\C A, \tau,\mathbf{P})$ in the sense of \cite[1.3.2.13]{hagII}. This consists, first of all, of an underlying combinatorial symmetric monoidal model category $\C V$, in the sense of \cite[Definition 4.2.6]{hmc}, and two full subcategories $\C V_0,\C A\subset\C V$ which intuitively play the role of the aisle and the heart of a $t$-structure in a triangulated category, respectively. Among the various required assumptions, the category $\comm(\C V)$ of \emph{commutative algebras} (i.e.~monoids) in $\C V$ must carry a model structure with fibrations and weak equivalences defined by the underlying morphisms in $\C V$. The model category $\aff_{\C V}=\comm(\C V)^\op$ of  \emph{affine stacks} is the opposite model category,  $\tau$ is a model pre-topology  on $\aff_{\C V}$, which induces a Grothendieck topology on $\ho(\aff_{\C V})$, and $\mathbf P$ is a class of morphisms  in $\aff_{\C V}$, which are regarded as good enough to define quotients. 

Apart from the axioms imposed in \cite[Definition 4.2.6]{hmc} to symmetric monoidal model categories, $\C V$ satisfies the strong unit axiom by \cite[Assumption 1.1.0.3]{hagII} and \cite[Lemma A.11]{htnso2}. The monoid axiom is not explicitly required in \cite{hagII}, but it is satisfied by all known examples. We will also assume that $\C V$ satisfies the monoid axiom, so it is a symmetric monoidal model category in the sense of Definition \ref{nota1}.

A \emph{simplicial presheaf} $F$ is a (contravariant) functor from affine stacks to simplicial sets,
$$F\colon \aff_{\C V}^\op\To\operatorname{Set}^{\Delta^\op},$$
i.e.~a (covariant) functor from $\comm(\C V)$.  The category $\operatorname{SPr}(\aff_{\C V})$ of simplicial presheaves carries a model structure where weak equivalences (resp.~fibrations) are pointwise weak equivalences (resp.~fibrations) of simplicial sets. In particular, a simplicial presheaf is fibrant iff its values are Kan complexes.

A \emph{stack} is a simplicial presheaf $F$ which preserves weak equivalences, finite homotopy products, and satisfies the following descent condition \cite[Corollary 1.3.2.4]{hagII}: given a commutative algebra $A$, a cosimplicial commutative algebra $B^\bullet$, and a cosimplicial map $A\r B^\bullet$ corresponding to a $\tau$-hypercover in $\aff_{\C V}$, where $A$ is regarded as a constant cosimplicial object, the induced morphism
$$F(A)\To \operatorname{holim}_{\deltan{n}\in\Delta}F(B^n)$$
is a weak equivalence of spaces.  The simplicial presheaf represented by a commutative algebra $A$, $$\rspecu (A)=\rmap_{\comm(\C V)}(A,-),$$ 
is a stack, that we call \emph{affine stack}. 
It is defined by a cosimplicial resolution of $A$ and a functorial fibrant resolution in $\comm(\C V)$. In order to turn $\rspecu$ into a Yoneda-like functor
$$\rspecu\colon\aff_\C{V}\To \operatorname{SPr}(\aff_{\C V})$$
we choose a functorial cosimplicial resolution in $\comm(\C V)$. 

The \emph{model category of stacks} $\aff_{\C V}^{\sim,\tau}$ is a left Bousfield localization of  $\operatorname{SPr}(\aff_{\C V})$ whose fibrant objects are stacks taking values in Kan complexes. The \emph{homotopy category of stacks} $\operatorname{St}(\C V,\tau)=\ho (\aff_{\C V}^{\sim,\tau})$ can be identified with the full subcategory of $\ho(\operatorname{SPr}(\aff_{\C V}))$  spanned by stacks. The homotopy limit of a diagram of stacks is the same in $\aff_{\C V}^{\sim,\tau}$ and in $\operatorname{SPr}(\aff_{\C V})$, i.e.~it can be computed as a pointwise homotopy limit in the category of spaces. 
A simplicial presheaf weakly equivalent to a stack in $\operatorname{SPr}(\aff_{\C V})$ is itself a stack. A map between stacks is a weak equivalence in $\aff_{\C V}^{\sim,\tau}$ if and only if it is a weak equivalence in $\operatorname{SPr}(\aff_{\C V})$. 
The functor $\rspecu$ above induces a full embedding
$$\rspecu\colon\ho(\aff_{\C V})\To \operatorname{St}(\C V,\tau).$$

Let $A$ be a commutative algebra in $\C V$. The category $\Mod{A}$ of $A$-modules is a model category with fibrations and weak equivalences defined by the underlying morphisms in $\C V$. In an ideal world, there would be a stack of modules, denoted by $\qcoh$, defined by 
$$\qcoh(A)=\abs{w\Mod{A}}.$$
Moreover, given a commutative algebra morphism $A\r B$, the induced map 
\begin{equation}\label{mala}
\qcoh(A)=\abs{w\Mod{A}}\To \qcoh(B)=\abs{w\Mod{B}}
\end{equation}
would be defined by the change of coefficients functor $-\otimes_AB\colon\Mod{A}\r \Mod{B}$. However, there are some issues. 

The homotopical issue is that the change of coefficients functor $-\otimes_AB$ does not preserve weak equivalences, so it cannot define a map as \eqref{mala}. This can be solved by restricting to full subcategories of cofibrant objects,
$$\qcoh(A)=\abs{w\Mod{A}_{c}},$$
since $-\otimes_AB$ is a left Quillen functor, so it preserves weak equivalences between cofibrant objects. This is not harmful since $\abs{w\Mod{A}_{c}}\simeq\abs{w\Mod{A}}$.

The categorical issue is that this constuction does not define a functor since,  given commutative algebra morphisms $A\r B\r C$ and an $A$-module $M$, the natural isomorphism $M\otimes_{A} B\otimes _{B}C\cong M\otimes_{A}C$ is not the identity. Nevertheless, these natural isomorphisms satisfy the coherence condition in \cite[D\'efinition I.56]{cmcla}, so the non-functor $A\mapsto\Mod{A}$ defines a weak presheaf of model categories, which can be strictified before restricting to weak equivalences and taking nerves, see  \cite[\S I.2.3.1]{cmcla}. The strictification of $\qcoh$ is explicitly described in  \cite[\S1.3.7]{hagII} in terms of categories of \emph{quasi-coherent modules}, hence the name $\qcoh$ of this simplicial presheaf.

The geometric issue is to check that the simplicial presheaf $\qcoh$ is a stack. This holds in any HAG context, essentially by definition, see \cite[Theorem 1.3.7.2]{hagII}.

Given an admissible operad $\mathcal O$ in $\C V$, in this section we aim at constructing a \emph{moduli stack of $\mathcal O$-algebras} $\algebrau{\mathcal O}{\C V}$. We would like that
$$\algebrau{\mathcal O}{\C V}(A)=\abs{w\algebra{\mathcal O}{\Mod{A}}},$$
and that given a commutative algebra morphism $A\r B$, the induced map 
\begin{equation*}
\algebrau{\mathcal O}{\C V}(A)=\abs{w\algebra{\mathcal O}{\Mod{A}}}\To 
\algebrau{\mathcal O}{\C V}(B)=\abs{w\algebra{\mathcal O}{\Mod{B}}}
\end{equation*}
were defined by the change of coefficients functor $-\otimes_AB$. We would also like that the forgetful functors $\algebra{\mathcal O}{\Mod{A}}\r\Mod{A}$ defined a morphism of stacks
$$\xi^{\mathcal{O}}\colon \algebrau{\mathcal{O}}{\C V}\To\qcoh.$$

In order to achieve this goal, we will face the same difficulties as for the definition of $\qcoh$ and some other specific issues. We will rely on the homotopy theory of operads developed in \cite{htnso,htnso2}.

Fist of all, in order the model category of $\mathcal O$-algebras in $\Mod{A}$ to be defined, we need to endow $\Mod{A}$ with a model $\C V$-algebra structure.

\begin{lem}\label{compatible0}
Let $A$ be a commutative algebra in $\C V$. The functor $$z_{A}=-\otimes A\colon \C V\r \Mod{A}$$ endows $\C C=\Mod{A}$ with the structure of a combinatorial symmetric model $\C V$-algebra in the sense of Definition \ref{nota2}. In addition, it satisfies the very strong unit axiom in Definition \ref{vsua}.
\end{lem}

\begin{proof}
By \cite[Assumption 1.1.0.2]{hagII},  $\Mod{A}$ is a combinatorial model category and satisfies the push-out product axiom and the unit axiom. The monoid axiom follows from \cite[Theorem 4.1 (2)]{ammmc}. The very strong unit axiom follows from \cite[Assumption 1.1.0.3]{hagII}, see Remark \ref{setiene}.

The functor $z_A$, which is obviously strong symmetric monoidal, is a left Quillen functor since its right adjoint, the forgetful functor $\Mod{A}\r\C V$, preserves fibrations and weak equivalences. Moreover, $z_A$ satisfies the $\unit$-cofibrant axiom by Lemma \ref{pseudowe}. 
\end{proof}

The next step is to show how, given a commutative algebra morphism $A\r B$, the change of coefficients functor $-\otimes_AB\colon\Mod{A}\r \Mod{B}$ induces a functor $-\otimes_{A}B\colon \algebra{\mathcal O}{\Mod{A}}\r  \algebra{\mathcal O}{\Mod{B}}$. 

\begin{defn}
A \emph{symmetric $\C V$-algebra functor} is a lax symmetric monoidal functor between symmetric $\C V$-algebras $F\colon\C C\r\C D$ equipped with a monoidal natural isomorphism $Fz_{\C C}(X)\cong z_{\C D}(X)$. 

A \emph{symmetric $\C V$-algebra Quillen adjunction} $F\colon \C C\rightleftarrows\C D\colon G$ is a  Quillen pair between symmetric model $\C V$-algebras which is a lax-lax symmetric monoidal adjunction and such that $F$ is equipped with the structure of a symmetric $\C V$-algebra functor. In particular $F$ is strong monoidal. We assume in addition that the \emph{$\unit$-cofibrant axiom} in \cite[Definition B.6]{htnso2} is satisfied.
\end{defn}

\begin{exm}\label{otroeje}
Given a commutative algebra morphism $A\r B$ in $\C V$, the functor $-\otimes_AB\colon\Mod{A}\r\Mod{B}$ is the left adjoint in a symmetric $\C V$-algebra Quillen adjunction. The natural isomorphism $z_A(X)\otimes_AB=X\otimes A\otimes_AB\cong X\otimes B=z_B(X)$ is the obvious one. The $\unit$-cofibrant axiom follows from Lemmas \ref{pseudowe} and \ref{compatible0}. If $A\r B$ is a weak equivalence then this symmetric $\C V$-algebra Quillen adjunction is a Quillen equivalence as a consequence of \cite[Assumption 1.1.0.3]{hagII}.
\end{exm}

\begin{rem}\label{otrorem}
Given a symmetric $\C V$-algebra Quillen adjunction $F\colon \C C\rightleftarrows\C D\colon G$ and an operad $\mathcal O$ in $\C V$, there is an induced Quillen pair between categories of algebras $F\colon \algebra{\mathcal{O}}{\C C}\rightleftarrows\algebra{\mathcal{O}}{\C D}\colon G$ which overlies the previous adjunction, see \cite[Proposition 7.1]{htnso2} and its proof. The latter is a Quillen equivalence if the former is and $\mathcal O$ is admissible \cite[Theorem D.11]{htnso2}.
\end{rem}

Therefore, if $\mathcal O$ is an admissible operad, we can define a simplicial presheaf $\algebrau{\mathcal O}{\C V}$ by 
\begin{equation}\label{midi}
\algebrau{\mathcal O}{\C V}(A)=\abs{w\algebra{\mathcal O}{\Mod{A}}_{c}}.
\end{equation}
Honestly speaking, we should go through a categorical strictification process before restricting to weak equivalences between cofibrant objects and taking nerves, as in the case of $\qcoh$, but we will keep this technical issue aside so as not to overload the paper.

Now we must show that $\algebrau{\mathcal O}{\C V}$ is a stack. The most complicated property is descent. We will see that  descent for operads follows from descent for modules, which is an assumption of HAG contexts. In order to check this technical condition, we need some definitions.

\begin{defn}
Given a cosimplicial commutative algebra $A^{\bullet}$, a \emph{cosimplicial $A^{\bullet}$-module} $M^{\bullet}=(M^{n},d^{i},s^{i})$ consists of a $A^{n}$-module $M^{n}$ for each $n\geq 0$ together with coface $d^{i}\colon M^{n-1}\r M^{n}$ and codegeneracy $s^{i}\colon M^{n+1}\r M^{n}$ morphisms of $A^{n}$-modules satisfying the usual cosimplicial identities. Here $M^{n-1}$ and $M^{n+1}$ are regarded as $A^{n}$-modules via restriction of scalars along the coface $d^{i}\colon A^{n-1}\r A^{n}$ and codegeneracy $s^{i}\colon A^{n+1}\r A^{n}$ in $A^{\bullet}$, respectively, $0\leq i\leq n$, 

A cosimplicial $A^{\bullet}$-module $M^{\bullet}$ is \emph{homotopy cartesian} if the derived adjoints $M^{n-1}\otimes^{\mathbb L}_{A^{n-1}} A^{n}\r M^{n}$ and $M^{n+1}\otimes^{\mathbb L}_{A^{n+1}} A^{n}\r M^{n}$ of all coface and codegeneracy maps are weak equivalences.
\end{defn}

\begin{rem}
The category $\Mod{A^{\bullet}}$ of cosimplicial $A^{\bullet}$-modules carries (at least) two combinatorial model structures, called \emph{projective} and \emph{injective}, with the same weak equivalences, compare \cite[Proposition A.2.8.2]{htt}. A morphism $f^{\bullet}\colon M^{\bullet}\r N^{\bullet}$ of cosimplicial $A^{\bullet}$-modules is a weak equivalence if each $f^{n}\colon M^{n}\r N^{n}$ is a 
weak equivalence. Moreover $f^{\bullet}$ is a \emph{projective fibration} if each $f^{n}$ is a fibration in $\Mod{A^{n}}$ (or equivalently in $\C V$), and $f^{\bullet}$ is an \emph{injective cofibration} if each $f^{n}$ is a cofibration in $\Mod{A^{n}}$. The injective model structure is technically more convenient, and will be the only one we use. To\"en and Vezzosi use the projective model structure in \cite{hagII}, probably because the injective model structure was not available at that time. We can freely change since both model structures have the same weak equivalences. 
\end{rem}

We must consider $\mathcal O$-algebras in $\Mod{A^{\bullet}}$. For this, we need the following result.

\begin{lem}
The injective model structure and the tensor product of $A^{n}$-modules, $n\geq 0$, induce a combinatorial symmetric monoidal model structure on $\Mod{A^{\bullet}}$ satisfying the very strong unit axiom. Moreover, 
the functor $z_{A^{\bullet}}\colon \C V\r \Mod{A^{\bullet}}$, $z_{A^{\bullet}}(X)^{n}=X\otimes A^{n}$, endows $\Mod{A^{\bullet}}$ with the structure of a symmetric model  $\C V$-algebra. 
\end{lem}

\begin{proof}
The push-out product axiom, the monoid axiom, and the very strong unit axiom in $\Mod{A^{\bullet}}$ follow from the corresponding axioms in $\Mod{A^{n}}$, checked in Lemma \ref{compatible0}, since we are working with the injective model structure. Hence $\Mod{A^{\bullet}}$ is a combinatorial symmetric monoidal model category in the sense of Definition \ref{nota1}, since it is locally presentable.

The functor $z_{A^{\bullet}}$ is obviously strong symmetric monoidal. Its right adjoint sends a cosimplicial $A^\bullet$-module $M^\bullet$ to the limit $\lim_{\deltan{n}\in\Delta}M^n$ in $\C V$. By Lemma \ref{compatible0}, each $z_{A^n}$ preserves (trivial) cofibrations, hence so does $z_{A^{\bullet}}$ for the injective model structure. Finally, the $\unit$-cofibrant axiom for $z_{A^{\bullet}}$ is a consequence of the fact that this axiom is satisfied by all $z_{A^n}$, see again Lemma \ref{compatible0}.
\end{proof}

Notice that the functor $z_{A^{\bullet}}$ sends cofibrant objects to homotopy cartesian simplicial $A^{\bullet}$-modules.

Now we are ready to show that $\algebrau{\mathcal{O}}{\C V}$ is a stack, at least when $\mathcal O$ is cofibrant.

\begin{prop}\label{esestack}
Let $\mathcal{O}$ be a cofibrant operad in $\C V$.
The simplicial presheaf $\algebrau{\mathcal{O}}{\C V}$ is a stack.
\end{prop}

In the proof of descent, we use the following technical observation.

\begin{lem}\label{categorico}
Let $F\colon \C A\rightleftarrows\C B\colon G$ be an adjoint pair $F\dashv G$. The functor $F$ is fully faithful if and only if the unit $X\r GF (X)$ is a natural isomorphism. Moreover, if $F$ is fully faithful then the essential image of $F$ consists of those objects $Y$ in $\C B$ such that the counit $FG(Y)\r Y$ is an isomorphism.
\end{lem}

\begin{proof}[Proof of Proposition \ref{esestack}]

By Example \ref{otroeje} and Remark \ref{otrorem}, if $A\st{\sim}\r B$ is a weak equivalence in $\comm(\C V)$ then $-\otimes_AB\colon\algebra{\mathcal{O}}{\Mod{A}}\r \algebra{\mathcal{O}}{\Mod{B}}$ is the left adjoint of a Quillen equivalence, hence it induces a weak equivalence $$\abs{w\algebra{\mathcal{O}}{\Mod{A}}_{c}}\st{\sim}\To \abs{w\algebra{\mathcal{O}}{\Mod{B}}_{c}}.$$ This proves that $\algebrau{\mathcal{O}}{\C V}$ preserves weak equivalences.

Let $A$ and $B$ be fibrant commutative algebras in $\C V$. The product category $\Mod{A}\times \Mod{B}$ is a symmetric model $\C V$-algebra with the obvious product structure. Moreover, the natural projections induce a symmetric model $\C V$-algebra Quillen equivalence \cite[proof of Lemma 1.3.2.3 (1)]{hagII},
$$\xymatrix@C=80pt{ \Mod{A\times B}\ar@<.5ex>[r]^-{(-\otimes_{A\times B} A,-\otimes_{A\times B} B)}& \Mod{A}\times \Mod{B}.\ar@<.5ex>[l]}$$ 
Again by Remark \ref{otrorem}, we obtain an induced Quillen equivalence between categories of $\mathcal O$-algebras,
$$\xymatrix{ \algebra{\mathcal O}{\Mod{A\times B}}\ar@<.5ex>[r]& \algebra{\mathcal O}{\Mod{A}\times \Mod{B}}=\algebra{\mathcal O}{\Mod{A}}\times \algebra{\mathcal O}{\Mod{B}}.\ar@<.5ex>[l]}$$
The left adjoint gives rise to a weak equivalence
$$\abs{\algebra{\mathcal O}{\Mod{A\times B}}_{c}}\st{\sim}\To\abs{\algebra{\mathcal O}{\Mod{A}}_{c}}\times \abs{\algebra{\mathcal O}{\Mod{B}}_{c}}.$$
Therefore $\algebrau{\mathcal{O}}{\C V}$ preserves finite homotopy products.

Now, let us tackle the descent condition. Let $A$ be a commutative algebra, that we also regard as a constant cosimplicial object, $B^{\bullet}$  a cosimplicial  commutative algebra, and  $\varphi\colon A\r B^{\bullet}$ a map as  in \cite[Assumption 1.3.2.2 (3)]{hagII}. There is a symmetric $\C V$-algebra Quillen adjunction
$$\xymatrix{ \Mod{A}\ar@<.5ex>[r]^-{\varphi_{*}}& \Mod{B^{\bullet}}\ar@<.5ex>[l]^{\varphi^{*}}}$$
defined by $\varphi_{*}(M)^{n}=M\otimes_{A}B^{n}$. This is indeed a Quillen pair because we are using the injective model structure on the right. The derived adjoint pair
$$\xymatrix{ \ho\Mod{A}\ar@<.5ex>[r]^-{\mathbb L\varphi_{*}}& \ho\Mod{B^{\bullet}}\ar@<.5ex>[l]^{\mathbb R\varphi^{*}}}$$
satisfyies the following properties: $\mathbb L\varphi_{*}$ is fully faithful and any homotopy cartesian $B^{\bullet}$-module is in the essential image of $\mathbb L\varphi_{*}$ \cite[Assumption 1.3.2.2 (3)]{hagII}. These two properties can be read as properties of the unit and counit of the derived adjoint pair, see Lemma \ref{categorico}.

By Remark \ref{otrorem}, the previous symmetric $\C V$-algebra Quillen adjunction induces a Quillen adjunction
$$\xymatrix{ \algebra{\mathcal O}{\Mod{A}}\ar@<.5ex>[r]^-{\varphi_{*}}& \algebra{\mathcal O}{\Mod{B^{\bullet}}}\ar@<.5ex>[l]^{\varphi^{*}}.}$$
Moreover, the derived adjunction
$$\xymatrix{ \ho\algebra{\mathcal O}{\Mod{A}}\ar@<.5ex>[r]^-{\mathbb L\varphi_{*}}& \ho\algebra{\mathcal O}{\Mod{B^{\bullet}}}\ar@<.5ex>[l]^{\mathbb R\varphi^{*}}}$$
overlies the previous derived adjoint pair, since cofibrant $\mathcal O$-algebras in $\Mod{A}$ have an underlying cofibrant $A$-module \cite[Corollaries C.3 and D.3]{htnso2}
and fibrant $\mathcal O$-algebras in $\Mod{B^{\bullet}}$ have an underlying fibrant cosimplicial $B^{\bullet}$-module. Therefore, the unit and counit of the former derived adjunction underlie those of the latter, so they share the properties in the statement of Lemma \ref{categorico}. This, together with \cite[Corollary B.0.8]{hagII}, finishes the proof.
\end{proof}

In order to check that $\algebrau{\mathcal{O}}{\C V}$ is actually a stack for any admissible operad $\mathcal O$, we need functoriality in $\mathcal O$ for $\algebrau{\mathcal{O}}{\C V}$. Continuing with our ideal picture, given a morphism of admissible operads $f\colon \mathcal O\r \mathcal P$, we would like to have a morphism of simplicial presheaves $$\algebrau{f}{\C V}\colon \algebrau{\mathcal{P}}{\C V}\To \algebrau{\mathcal{O}}{\C V}$$ defined by the contravariant change of base operad functor $f^{*}$ in \eqref{qp},
$$\algebrau{\mathcal{P}}{\C V}(A)=|w\algebra{\mathcal{P}}{\Mod{A}}|\To  \algebrau{\mathcal{O}}{\C V}(A)=|w\algebra{\mathcal{O}}{\Mod{A}}|.$$
The functor $f^{*}$ does preserve all weak equivalences. The problem here is that, in order to define the simplicial presheaf $\algebrau{\mathcal{O}}{\C V}$, we have had to restrict to cofibrant $\mathcal O$-algebras \eqref{midi}, and $f^{*}$ does not preserve cofibrant objects.

In order to solve this problem, we thicken $\algebrau{\mathcal{O}}{\C V}$, i.e.~we give an alternative but weakly equivalent definition in terms of nerves of categories of weak equivalences between a class of objects bigger than cofibrant ones.

Let $\algebrapc{\mathcal{O}}{\Mod{A}}\subset \algebra{\mathcal{O}}{\Mod{A}}$ be the full subcategory of  $\mathcal O$-algebras whose  underlying $A$-module is pseudo-cofibrant in the sense of \cite[Definition A.1]{htnso2}, see also Remark \ref{sua} below. The definition of pseudo-cofibrant object is recalled in Remark \ref{sua} below. Cofibrant algebras over an admissible operad $\mathcal O$ have a pseudo-cofibrant underlying $A$-module by \cite[Corollary D.3]{htnso2}, i.e.~$\algebra{\mathcal{O}}{\Mod{A}}_{c}\subset \algebrapc{\mathcal{O}}{\Mod{A}}$. Moreover, this inclusion induces a weak equivalence on nerves  $|w\algebra{\mathcal{O}}{\Mod{A}}_{c}|\simeq |w\algebrapc{\mathcal{O}}{\Mod{A}}|$, since functorial cofibrant replacements define a homotopy inverse.

We can equivalently define $\algebrau{\mathcal{O}}{\C V}$ by
$$\algebrau{\mathcal{O}}{\C V}(A)=|w\algebrapc{\mathcal{O}}{\Mod{A}}|$$
since, given a commutative algebra morphism $A\r B$, the change of coefficients functor $-\otimes_AB$ preserves pseudo-cofibrant modules, see \cite[Assumption 1.1.0.2]{hagII} and \cite[Remark B.9]{htnso2}, and weak equivalences between them, see Lemma \ref{pseudowe}. 

The functor $f^{*}$ in \eqref{qp} is the identity on underlying objects, hence it yields a  morphism of simplicial presheaves $\algebrau{f}{\C V}\colon \algebrau{\mathcal{P}}{\C V}\r \algebrau{\mathcal{O}}{\C V}$ given by
$$\algebrau{\mathcal{P}}{\C V}(A)=|w\algebrapc{\mathcal{P}}{\Mod{A}}|\To  \algebrau{\mathcal{O}}{\C V}(A)=|w\algebrapc{\mathcal{O}}{\Mod{A}}|.$$
Moreover, this morphism is a weak equivalence if $f$ is a weak equivalence.

\begin{cor}
Let $\mathcal{O}$ be an admissible operad in $\C V$. The simplicial presheaf $\algebrau{\mathcal{O}}{\C V}$ is a stack.
\end{cor}

\begin{proof}
Let $f\colon \widetilde{\mathcal O}\st{\sim}\onto \mathcal O$ be a cofibrant replacement. The morphism of simplicial presheaves
$ \algebrau{f}{\C V}\colon \algebrau{\mathcal{O}}{\C V}\r \algebrau{\widetilde{\mathcal{O}}}{\C V}$
is a weak equivalence, and $\algebrau{\widetilde{\mathcal{O}}}{\C V}$ is a stack by Proposition \ref{esestack}. Hence $\algebrau{\mathcal{O}}{\C V}$ is also a stack.
\end{proof}

\begin{rem}\label{hococo}
By Corollary \ref{rezkcor}, $\algebrau{\mathcal O }{\C V}$, regarded as a contravariant functor in $\mathcal O$, takes homotopy colimits of admissible operads to homotopy limits of stacks.
\end{rem}

For $\mathcal O$ an admissible operad, the forgetful functors $\algebra{\mathcal{O}}{\Mod{A}}\r \Mod{A}$ induce a morphism of stacks
$$\xi^{\mathcal{O}}\colon \algebrau{\mathcal{O}}{\C V}\To\qcoh.$$
Indeed, if $\mathcal O(0)$ is cofibrant in $\C V$, e.g.~if $\mathcal O$ is a cofibrant operad \cite[Corollary C.3]{htnso2}, we can use the unthickened definition of $\algebrau{\mathcal{O}}{\C V}$ to define this morphism  by the maps
$$\algebrau{\mathcal{O}}{\C V}(A)=|w\algebra{\mathcal O}{\Mod{A}}_{c}|\To\qcoh(A)=|w\Mod{A}_{c}|,$$
since  the underlying object of a cofibrant $\mathcal O$-algebra  is cofibrant \cite[Corollary C.3]{htnso2}. 
For a general $\mathcal O$, we thicken $\qcoh$ without changing its homotopy type by using the subcategories $\Modpc{A}\subset\Mod{A}$ of pseudo-cofibrant $A$-modules,
$$\qcoh(A)=|w\Modpc{A}|.$$
Then we define $\xi^{\mathcal{O}}$ by using the thickened versions of the source and the target stacks, and the maps
$$\algebrau{\mathcal{O}}{\C V}(A)=|w\algebrapc{\mathcal O}{\Mod{A}}|\To\qcoh(A)=|w\Modpc{A}_{c}|.$$
The morphism $\xi^{\mathcal{O}}$ is natural in the admissible operad $\mathcal O$. Our next aim is to study properties of this morphism.

\begin{defn}
Let $\mathcal O$ be an admissible operad in $\C V$, $A$ a commutative algebra in $\C V$, and $M$ an $A$-module. The \emph{stack of $\mathcal O$-algebra structures on $M$} is the stack over $\rspecu(A)$
$$\rmapu_{\operad{\C V}}(\mathcal{O},\mathtt{End}_{\Mod{A}}(M))$$
defined as the homotopy pull-back of $\xi^{\mathcal{O}}$ along the morphism $g\colon \rspecu(A)\rightarrow \qcoh$ represented by  $M$. 
\end{defn}

Honestly speaking, we should  take a fibrant replacement of $\qcoh$ in order the map $g$ to be defined in the category of stacks, not just in the homotopy category. However, we will allow us this kind of language.

\begin{rem}\label{overa}
For any cofibrant commutative algebra $A$ in $\C V$, $\Mod{A}$ is the base combinatorial symmetric monoidal model category of a new HAG context and $\aff_{\Mod{A}}^{\sim,\tau}$ is Quillen equivalent to the comma model category $\aff_{\C V}^{\sim,\tau}\downarrow\rspecu(A)$, see \cite[Proposition 1.3.2.10]{hagII} and its preceding paragraph. Hence, we can identify stacks over 
$\rspecu(A)$ with stacks in this new HAG context.
\end{rem}

\begin{rem}\label{vayavaya}
By Theorem \ref{rezk}, for any $A$-algebra $B$,
$$\rmapu_{\operad{\C V}}(\mathcal{O},\mathtt{End}_{\Mod{A}}(M))(B)\simeq
\rmap_{\operad{\C V}}(\mathcal{O},\mathtt{End}_{\Mod{B}}(\widetilde{M\otimes_{A}^{\mathbb L}B})),$$
where $\widetilde{N}$ denotes a fibrant-cofibrant replacement of a $B$-module $N$. This justifies the name of this stack.
\end{rem}

\begin{rem}\label{yayaya}
The stack $\rmapu_{\operad{\C V}}(\mathcal{O},\mathtt{End}_{\Mod{A}}(M))$ is a contravariant functor in the admissible operad $\mathcal O$ since $\algebrau{\mathcal{O}}{\C V}$ is a contravariant functor in $\mathcal O$ and $\xi^{\mathcal O}$ is natural. Moreover, it takes homotopy colimits of admissible operads to homotopy limits of stacks, since mapping spaces take  homotopy colimits in the first variable to homotopy limits of spaces.
\end{rem}

We will show that the stack of $\mathcal O$-algebra structures on a module $M$ is representable when $M$ is perfect in the sense of the following definition.

\begin{defn}[{\cite[Definition 1.2.3.6]{hagII}}]\label{perfecto}
Let $Y$ be an object in $\C V$. For the sake of simplicity, let us assume that $Y$ is fibrant and cofibrant. Let 
\begin{align*}
q_X\colon &QX\st{\sim}\onto X,&
r_X\colon &X\st{\sim}\into RX,
\end{align*}
be functorial cofibrant and fibrant replacements in $\C V$, respectively.
Moreover, denote internal morphism objects in $\C V$ by $\hom_{\C V}$. 

The  \emph{dual} of $Y$  is $Y^\vee=\hom_{\C V}(Y,R\unit)$. 
The object $Y$ is \emph{perfect} if the composition of the two vertical morphisms in the following diagram is a weak equivalence in $\C V$,
\begin{equation*}
\xymatrix@C=7pt{ 
Y\otimes Y^\vee\ar@{=}[r]&\hom_{\C V}(\unit,Y)\otimes\hom_{\C V}(Y,R\unit)\ar[rr]^-\sim_-{q_{\unit}^{*}\otimes\id{}}&&
\hom_{\C V}(Q\unit,Y)\otimes\hom_{\C V}(Y,R\unit)\ar[d]^-\otimes\\
&&&\hom_{\C V}(Q\unit\otimes Y,Y\otimes R\unit)\ar[d]^{(r_{Y\otimes R\unit})_{*}}\\
\hom_{\C V}(Y,Y)\ar@{=}[r]&
\hom_{\C V}(\unit\otimes Y,Y\otimes\unit)\ar[rr]^-\sim_-{\begin{array}{c}\\[-10pt]\scriptstyle\hom_{\C V}(q_{\unit}\otimes\id{Y},r_{Y\otimes R\unit}(\id{Y}\otimes r_{Y}))\end{array}}
&&\hom_{\C V}(Q\unit\otimes Y,R(Y\otimes R\unit))
}
\end{equation*}
In this case, this diagram yields an isomorphism $Y\otimes Y^{\vee}\cong \hom_{\C V}(Y,Y)$ in $\ho\C V$. An arbitrary $Y$ is perfect if a fibrant-cofibrant replacement is perfect.

If $A$ is a commutative algebra, we can replace $\C V$ with $\Mod{A}$ and $\unit$ with $A$ in the previous paragraphs. This yields the notions of dual and perfect $A$-module.
\end{defn}

\begin{thm}\label{lageo}
Let $\mathcal{O}$ be an admissible  operad in $\C V$, $A$ a commutative algebra in $\C V$ and $M$ a perfect $A$-module. The stack $\rmapu_{\operad{\C V}}(\mathcal{O},\mathtt{End}_{\Mod{A}}(M))$ is affine.
\end{thm}

We need a technical result in order to prove this theorem. Recall that a \emph{sequence}  is an object in the product model category $\C V^{\mathbb N}=\prod_{n\geq 0}\C V$. There is a Quillen pair
$$\xymatrix{\C V^{\mathbb N}\ar@<.5ex>[r]^-{\mathcal F}&\ar@<.5ex>[l]^-{\text{forget}}\operad{\C V}}$$
where the right adjoint is the forgetful functor $\mathcal O\mapsto\{\mathcal O(n)\}_{n\geq 0}$ and the left adjoint $\mathcal F$ is the \emph{free operad} functor \cite[\S5]{htnso}. 

\begin{lem}\label{horror}
Any operad in $\C V$ is weakly equivalent to the target of a relative cell complex from a free operad on a cofibrant sequence with respect to a set of cofibrations between free operads on cofibrant sequences.
\end{lem}

\begin{proof}
Any operad is weakly equivalent to a cell complex with respect to the set of generating cofibrations in $\operad{\C V}$. The initial operad is free on a cofibrant sequence: the initial sequence. Generating cofibrations are maps between free operads, actually they are free maps, see the proof of \cite[Theorem 1.1]{htnso}. The underlying sequences of these free operads need not be cofibrant. However, if the tensor unit $\unit$ is cofibrant, we can apply the trick at the beginning of the proof of \cite[Proposition 4.2]{htnso2} to use  cofibrant sequences instead. 

If the tensor unit is not cofibrant, the previous trick only allows us to use sequences $U=\{U(n)\}_{n\geq 0}$ with $U(n)$ cofibrant for $n\neq 1$ and $\unit$-cofibrant for $n=1$, see \cite[Corollary C.3]{htnso}. Such sequences can be replaced with  cofibrant sequences as follows. Let $\tilde U\st{\sim}\onto U$ be a cofibrant resolution in $\C V^{\mathbb N}$. The induced morphism $\mathcal F(\tilde U)\r\mathcal F(U)$ is a weak equivalence by \cite[Corollary A.14 and Lemma A.15]{htnso2}, since  a free operad is built from  coproducts of tensor products of the objects in the underlying sequence, see \cite[\S5]{htnso}. Finally, we can apply the gluing lemma in the cofibration category of operads with underlying pseudo-cofibrant sequence to replace $\mathcal F(U)$ with $\mathcal F(\tilde U)$, see \cite[Proposition C.8 and Corollary 5.2]{htnso2}. 
\end{proof}

\begin{proof}[Proof of Theorem \ref{lageo}]

By Remark \ref{yayaya} and Lemma \ref{horror}, it is enough to prove the theorem for $\mathcal O=\mathcal F(U)$ a free operad on a cofibrant sequence $U=\{U_{n}\}_{n\geq 0}$.

We can suppose without loss of generality that $A$ is a cofibrant commutative algebra and $M$ a fibrant-cofibrant $A$-module. We  define a stack $F$ 
over $\rspecu(A)$ in the sense of Remark \ref{overa}, as follows. Given a commutative $A$-algebra $B$, i.e.~a morphism of commutative algebras $A\r B$,
\begin{align*}
F(B)
&=
\prod_{n\geq0}\rmap_{\Mod{B}}(U_{n}\otimes M^{\otimes_A n}\otimes_{A}M^{\vee}, B).
\end{align*}
This stack is clearly represented by the free commutative $A$-algebra $C$ generated by
$$\coprod_{n\geq 0}U_{n}\otimes M^{\otimes_A n}\otimes_A M^{\vee}.$$

Notice that, by adjunction 
$$
F(B)\simeq
\rmap_{\operad{\C V}}(\mathcal{F}(U),\mathtt{End}_{\Mod{B}}(\widetilde{M\otimes_AB})),$$
where $\widetilde{M\otimes_AB}$ denotes a fibrant replacement of the cofibrant $B$-module $M\otimes_AB$, see \cite[Proposition 1.2.3.7]{hagII} and its preceding paragraphs. However,  we cannot define $F$ by the right hand side of the previous equation since it 
is not functorial in $B$.

Theorem \ref{rezk} shows that, pointwise, $\rspecu(C)\simeq F$ is the homotopy pull-back of $\xi^{\mathcal F(U)}$ along $g$. It is only left to define a map $\rspecu(C)\r \algebrau{\mathcal O }{\C V}$ realizing the pointwise weak equivalence. One can straightforwardly check that it is enough to take a map determined by an $\mathcal O$-algebra structure on the $C$-module $\widetilde{M\otimes_A C}$ corresponding to the connected component of the identity in $\rmap_{\operad{\C V}}(\mathcal{F}(U),\mathtt{End}_{\Mod{C}}(\widetilde{M\otimes_AC}))\simeq\rspecu(C)(C)=\rmap_{\comm(\C V)}(C,C)$.

\end{proof}

One application of Theorem \ref{lageo} is to obtain geometric substacks of $\algebrau{\mathcal O}{\C V}$. 
The stack $\qcoh$ is too big to have nice geometric properties. To this end, it is necessary to restrict to smaller substacks.

\begin{defn}\label{substack}
Given a stack $G$, a \emph{substack} of $G$ is a stack $F$ such that, for any $X$ in $\aff_{\C V}$, $F(X)\subset G(X)$ is a simplicial subset formed by certain connected components of $G(X)$ and the inclusions define a morphism $F\r G$. 
\end{defn}

\begin{exm}\label{1geo1}
Below we give examples of $1$-geometric substacks $F\subset \qcoh$ in different HAG contexts. In each case, we give a name for $F$ and specify the property that an $A$-module must satisfy so that the corresponding component of $\qcoh(A)\simeq|w\Mod{A}|$ lies in $F(A)$.

\begin{enumerate}
\item The substack $\underline{\operatorname{Vect}}_n$ of \emph{rank $n$ vector bundles}, $n\geq 0$, for any HAG context where the tensor unit $\unit$ is finitely presented in $\C V$ in the sense of \cite[Definition 1.2.3.1]{hagII}, and such that all smooth morphisms of commutative algebras belong to $\mathbf P$, see \cite[Corollary 1.3.7.12]{hagII}. An $A$-module $M$ is a rank $n$ vector bundle if there is a covering $\{A\r A_i\}_{i\in I}$ in $\tau$ such that each $M\otimes^{\mathbb L}_AA_i $ is weakly equivalent to the (derived) product $A_i\times^{h}\st{n}\cdots\times^{h} A_i$.

\item The substack $\perf$ of \emph{perfect modules} in the sense of Definition \ref{perfecto}, in the weak HAG context for complicial algebraic geometry \cite[2.3.2]{hagII}, see \cite[Proposition 2.3.3.1]{hagII}.

\item The substack $\perf_{[a,b]}$ of \emph{locally cellular modules of amplitude contained in a finite interval $[a,b]$} in the sense of \cite[Definition 2.3.5.2]{hagII}, in the stronger HAG context for complicial algebraic geometry \cite[2.3.4]{hagII}, see \cite[Proposition 2.3.5.4]{hagII}.
\end{enumerate}

Example (1) still holds if $\unit$ is not finitely presented but $\mathbf P$ contains all formally smooth morphisms. 
Examples (2) and (3) have versions in brave new algebraic geometry, see \cite[2.4.1]{hagII}. 
\end{exm}

\begin{defn}\label{1geo1.5}
Let $\mathcal O$ be an admissible operad in $\C V$ and  $F$ a substack of $\qcoh$. We define the \emph{restricted stack of $\mathcal O$-algebras} $\algebrau{\mathcal O}{F}$ as the following pull-back,
$$\xymatrix{
\algebrau{\mathcal O}{F}\ar[r]\ar[d]\ar@{}[rd]|-{\text{pull}}&\algebrau{\mathcal O}{\C V}\ar[d]^-{\xi^{\mathcal O}}\\
F\ar[r]_-{\text{incl.}}&\qcoh.
}$$
\end{defn}

\begin{rem}
The pull-back square in Definition \ref{1geo1.5} is a homotopy pull-back since the bottom horizontal arrow is a fibration of simplicial presheaves for obvious reasons. Therefore $\algebrau{\mathcal O}{F}$ is indeed a stack.
\end{rem}

\begin{prop}\label{1geo2}
In the conditions of Definition \ref{1geo1.5}, if  $F$ is $n$-geometric and  for any commutative algebra $A$ in $\C V$, the connected components of $F(A)$ are represented by perfect $A$-modules, then $\algebrau{\mathcal O}{F}$ is $n$-geometric.
\end{prop}

\begin{proof}
Theorem \ref{lageo} shows that $\xi^{\mathcal O}_{F}\colon \algebrau{\mathcal O}{F}\r F$ is an affine morphism. Hence $\algebrau{\mathcal O}{F}$ is $n$-geometric by  \cite[Proposition 1.3.3.4]{hagII}.
\end{proof}

\section{Spaces and stacks of (unital) $A$-infinity algebras}\label{final}

Let $\C V$ be a symmetric monoidal model category. Assume further that $\C V$ is simplicial or complicial, i.e.~a symmetric model $\operatorname{Set}^{\Delta^{\op}}$-algebra or $\operatorname{Ch}(\Bbbk)$-algebra. Recall from Remark \ref{uass0} the canonical map $\phi\colon\mathtt{Ass}\r \mathtt{uAss}$ from the associative operad to the unital associative operad which models the forgetful functor from unital associative algebras to associative algebras. We proved in \cite{udga} that  $\phi$ is a homotopy epimorphism in $\operad{\C V}$ in the sense of the following definition, compare \cite[Remark 1.2.6.2]{hagII} and \cite[\S2]{udga}.

\begin{defn}
A map of simplicial sets $g\colon K\r L$ is a \emph{homotopy monomorphism} if it corestricts to a weak equivalence between $K$ and a subset of connected components of $L$. This is equivalent to say that the homotopy fibers of $g$ are empty or weakly contractible. 

A morphism $f\colon X\r Y$ in a model category $\C M$ is a \emph{homotopy epimorphism} if for any object $Z$ in $\C M$ the induced map $f^{*}\colon \rmap_{\C M}(Y,Z)\r  \rmap_{\C M}(X,Z)$ is a homotopy monomorphism. 

A map $f\colon X\r Y$ is a \emph{homotopy monomorphism} in $\C M$ if it is a homotopy epimorphism in $\C M^{\op}$. This is the same as saying that the derived codiagonal, which is the map $\Delta=\binom{1_{X}}{1_{X}}\colon X\r X\times_{Y}^{h}X$ to the homotopy product of $X$ with itself over $Y$, is a weak equivalence. 

The two notions of homotopy monomorphism of simplicial sets coincide.
\end{defn}

In particular, if $\C C$ is a model $\C V$-algebra and $Y$ is a fibrant-cofibrant object in $\C C$, the map
\begin{equation}\label{momo}
\phi^{*}\colon \rmap_{\operad{\C V}}(\mathtt{uAss},\mathtt{End}_{\C C}(Y))\To  \rmap_{\operad{\C V}}(\mathtt{Ass},\mathtt{End}_{\C C}(Y))
\end{equation}
is a homotopy monomorphism. 

\begin{prop}\label{momomo}
The map $|w\algebra{\mathtt{uAss}}{\C C}|\r |w\algebra{\mathtt{Ass}}{\C C}|$ induced by the forgetful functor is a homotopy monomorphism.
\end{prop}

This result is a consequence of Theorem \ref{rezk} and Lemma \ref{porfibras} below. It also  follows from \cite[Theorem 5.2.3.5]{lurieha}, which actually goes further, characterizing the essential image, see Proposition \ref{lurie} below. 

\begin{lem}\label{porfibras}
Let $$\xymatrix@R=10pt{K\ar[rd]^{f}\ar[dd]_{g}&\\
&M\\
L\ar[ru]_{h}}$$
be a commutative triangle of simplicial sets. Denote $F_{f,x}$ and $F_{h,x}$ the homotopy fibers of $f$ and $h$ at $x\in M_{0}$, respectively. The following statements are equivalent:
\begin{enumerate}
\item $g$ is a homotopy monomorphism
\item For any $x\in M_{0}$ the  map $g_{x}\colon F_{f,x}\r F_{h,x}$ induced by $g$ is a homotopy monomorphism. 
\end{enumerate}
\end{lem}

\begin{proof}
The square
$$\xymatrix{F_{f,x}\ar[r]\ar[d]_{g_{x}}&K\ar[d]^{g}\\
F_{h,x}\ar[r]&L}$$
is a homotopy pull-back. This lemma follows from the fact that parallel arrows in a homotopy pull-back have essentially the same homotopy  fibers. Let us spell out what this means in this case. We are interested in the vertical arrows. The homotopy fiber of $g_{x}$ at a vertex $y$ of $F_{h,x}$ coincides with the homotopy fiber of $g$ at the image of $y$ along $F_{h,x}\r L$. Moreover, the homotopy fiber of $g$ at a vertex $z$ of $L$ coincides with the homotopy fiber of $g_{h(z)}$ at any vertex $y$ of $F_{h,h(z)}$ which maps to the same component as $z$ in $L$.
\end{proof}

The image of the injective map $\pi_{0}|w\algebra{\mathtt{uAss}}{\C C}|\hookrightarrow \pi_{0}|w\algebra{\mathtt{Ass}}{\C C}|$ has a friendly characterization.

\begin{defn}\label{qu}
An  associative algebra $X$ in $\C C$ with underlying fibrant-cofibrant object is \emph{quasi-unital} if
there exists a map $u\colon \tilde\unit\r X$, where $q_{\unit}\colon \tilde\unit\st{\sim}\onto \unit$ is a cofibrant replacement of the tensor unit, such that the  maps 
$$\begin{array}{c}
\tilde\unit\otimes X\st{u\otimes X}\To X\otimes X\st{\text{mult.}}\To X,\\
X\otimes \tilde\unit\st{X\otimes u}\To X\otimes X\st{\text{mult.}}\To X,
\end{array}$$
are homotopic to $\tilde\unit\otimes X\mathop{\To}\limits^{q_{\unit}\otimes X}_{\sim}\unit\otimes X\cong X$ and $X\otimes\tilde\unit\mathop{\To}\limits^{X\otimes q_{\unit}}_{\sim}\unit\otimes X\cong X$, respectively. 

If the associative algebra $X$ does not have an underlying fibrant-cofibrant object, we say that $X$ is quasi-unital if a (and hence any) fibrant-cofibrant replacement of $X$ in $\algebra{\mathtt{Ass}}{\C C}$ is quasi-unital. Such a fibrant-cofibrant replacement has an underlying fibrant-cofibrant object in $\C C$ by \cite[Corollary D.3]{htnso2}.
\end{defn}

\begin{prop}\label{lurie}
The image of $\pi_{0}|w\algebra{\mathtt{uAss}}{\C C}|\hookrightarrow \pi_{0}|w\algebra{\mathtt{Ass}}{\C C}|$ consists of the connected components of quasi-unital associative algebras.
\end{prop}

This proposition follows from \cite[Theorem 5.2.3.5]{lurieha}, which is a deep and complicated result. We have not found any elementary proof for this proposition. There is an elementary proof in case $\C V=\C C=\operatorname{Ch}(\Bbbk)$ is the category of chain complexes over a field $\Bbbk$ and $z$ is the identity functor. This proof uses the existence of minimal models for $A$-infinity algebras, compare \cite[Remark 6.8]{htnso2}. The result for $\Bbbk$ an arbitrary commutative ring, first proved in \cite{uainfcat}, is already very complicated.

We obtain as a corollary a similar characterization of the image of $\pi_{0}\eqref{momo}$. Recall that vertices in $\rmap_{\operad{\C V}}(\mathtt{Ass},\mathtt{End}_{\C C}(Y))$ do not correspond to associative algebra structures on $Y$, but to $A$-infinity algebra structures, i.e.~algebra structures over a cofibrant resolution $\mathtt{A}_{\infty}\st{\sim}\onto\mathtt{Ass}$ in $\operad{\C V}$ of the associative operad $\mathtt{Ass}$.

\begin{cor}\label{luriecor}
Let $Y$ be a fibrant-cofibrant objetc in $\C C$. 
The image of the injective map $\pi_{0}\rmap_{\operad{\C V}}(\mathtt{uAss},\mathtt{End}_{\C C}(Y))\hookrightarrow \pi_{0} \rmap_{\operad{\C V}}(\mathtt{Ass},\mathtt{End}_{\C C}(Y))$ consists of the connected components of $A$-infinity algebra structures on $Y$ which are weakly equivalent in $\algebra{\mathtt{A}_{\infty}}{\C C}$ to a quasi-unital associative algebra.
\end{cor}

This result follows from Proposition \ref{lurie}, Theorem \ref{rezk}, and an elementary computation with the low-dimensional part of the long homotopy exact sequence of a homotopy fibration.

\begin{rem}\label{estricto}
Since we can strictify algebras over cofibrant resolutions of admissible operads by using the Quillen equivalence \eqref{qp}, 
Corollary \ref{luriecor} gives a postive answer to the question raised in \cite[Remark 7.5]{htnso2}.
\end{rem}

Quasi-unital associative algebras can be characterized in terms of operads.

\begin{defn}\label{assqu}
Given an object $U$ in $\C V$ and an integer $m\geq 0$, denote $U[m]$ the sequence of objects in $\C V$ whose $m^{\text{th}}$ term is $U$ and whose other terms are the initial object $\varnothing$.  This construction defines a functor $\C V\r\C V^{\mathbb N}$, left adjoint to $\{V(n)\}_{n\geq 0}\mapsto V(m)$. 

Let $q_{\unit}\colon \tilde\unit\st{\sim}\onto \unit$ be a cofibrant replacement of the tensor unit. Choose a cylinder for $\tilde \unit$, which is a factorization of the folding map
$$\tilde \unit\coprod\tilde \unit\st{i}\into C\mathop{\onto}^{p}_{\sim} \tilde \unit.$$

The operad $\mathtt{Ass}^{\operatorname{qu}}$ is defined as the following push-out in $\operad{\C V}$,
$$\xymatrix@C=80pt{
\mathcal F((\tilde \unit\coprod\tilde \unit)[1])\coprod\mathcal F((\tilde \unit\coprod\tilde \unit)[1])\ar@{>->}[r]^-{\mathcal F(i[1])\coprod \mathcal F(i[1])}\ar[d]_{(c_{1},c_{2})}\ar@{}[rd]|{\text{push}}& \mathcal F(C[1])\coprod\mathcal F(C[1])\ar[d]\\
\mathtt{Ass}\coprod\mathcal F(\tilde\unit[0])\ar@{>->}[r]&\mathtt{Ass}^{\operatorname{qu}}
}$$
The morphism $c_{1}$ is defined by a map $(c_{11},c_{12})\colon \tilde \unit\coprod\tilde\unit\r (\mathtt{Ass}\coprod\mathcal F(\tilde\unit[0]))(1)$. The map $c_{11}$ is the composite
$$\tilde \unit \mathop{\onto}^{q_{\unit}}_{\sim}\unit=\mathtt{Ass}(1)\To (\mathtt{Ass}\coprod\mathcal F(\tilde\unit[0]))(1),$$
where the second map is given by the inclusion of the first factor of the coproduct. The map $c_{12}$ is the composite
$$\xy
(-6.3,0)*{\tilde\unit\cong \unit\otimes\tilde\unit=\mathtt{Ass}(2)\otimes \mathcal F(\tilde\unit[0])(0)}="a",
(0,-10)*+{(\mathtt{Ass}\coprod\mathcal F(\tilde\unit[0]))(2)\otimes
(\mathtt{Ass}\coprod\mathcal F(\tilde\unit[0]))(0)}="b",
(0,-20)*+{(\mathtt{Ass}\coprod\mathcal F(\tilde\unit[0]))(1).}="c"
\ar(0,-3);"b"
\ar"b";"c"^{\circ_{1}}
\endxy$$
Here the first arrow is given by the tensor product of the inclusions of the factors of the coproduct. Similarly $c_{2}$ is defined by a map $(c_{21},c_{22})$ such that $c_{21}=c_{11}$ and $c_{22}$ is defined like $c_{12}$ replacing $\circ_{1}$ with $\circ_{2}$.
\end{defn}

Since $\tilde\unit$ is cofibrant, the inclusion of the first factor $\mathtt{Ass}\into\mathtt{Ass}\coprod\mathcal F(\tilde\unit[0])$ is a cofibration, and so is the composite $\mathtt{Ass}\into \mathtt{Ass}^{\operatorname{qu}}$. In particular, $\mathtt{Ass}^{\operatorname{qu}}$ is admissible by \cite[Corollary C.2]{htnso2}. More precisely, $\mathtt{Ass}^{\operatorname{qu}}(0)$ is cofibrant and $\mathtt{Ass}^{\operatorname{qu}}(n)$ is $\unit$-cofibrant for all $n>0$.

\begin{lem}\label{yiyiyi}
Let $X$ be an associative algebra in $\C C$ with underlying fibrant-cofibrant object. Then $X$ is quasi-unital if and only if the map $\mathtt{Ass}\r\mathtt{End}_{\C C}(X)$ defining the associative algebra structure factors as $\mathtt{Ass}\into \mathtt{Ass}^{\operatorname{qu}}\r\mathtt{End}_{\C C}(X)$.
\end{lem}

This lemma follows directly from Definitions \ref{qu} and \ref{assqu}. 

%

The universal property of a push-out shows the existence of a unique map $$\psi\colon\mathtt{Ass}^{\operatorname{qu}}\To \mathtt{uAss}$$ such that the composite $\mathtt{Ass}\into \mathtt{Ass}^{\operatorname{qu}}\st{\psi}\r \mathtt{uAss}$ is $\phi$,  defined by 
\begin{align*}
\tilde \unit&\mathop{\onto}\limits^{q_{\unit}}_{\sim}\unit=\mathtt{uAss}(0),&
C&\mathop{\onto}\limits^{p}_{\sim}\tilde \unit\mathop{\onto}\limits^{q_{\unit}}_{\sim}\unit=\mathtt{uAss}(1),
\end{align*}
on both copies of $C$.

\begin{prop}\label{retracto}
For any fibrant-cofibrant object $Y$ in $\C C$, the map 
$$\psi^{*}\colon
\rmap_{\operad{\C V}}(\mathtt{uAss},\mathtt{End}_{\C C}(Y))\To
\rmap_{\operad{\C V}}(\mathtt{Ass}^{\operatorname{qu}},\mathtt{End}_{\C C}(Y))$$
admits a retraction in the homotopy category of simplicial sets.
\end{prop}

\begin{proof}
The composite
$$
\xymatrix{\rmap_{\operad{\C V}}(\mathtt{uAss},\mathtt{End}_{\C C}(Y))\ar[d]^{\psi^{*}}\\
\rmap_{\operad{\C V}}(\mathtt{Ass}^{\operatorname{qu}},\mathtt{End}_{\C C}(Y))\ar[d]\\
\rmap_{\operad{\C V}}(\mathtt{Ass},\mathtt{End}_{\C C}(Y))}$$
is the homotopy monomorphism in \eqref{momo}, which it induces a weak equivalence between the source and a subset of connected components of the target. By Corollary \ref{luriecor} and Lemma \ref{yiyiyi}, the image of the second map lies on those connected components. Here we use that we can strictify algebras over cofibrant resolutions of admissible operads by using the Quillen equivalence \eqref{qp}. Hence we are done.
\end{proof}


Now, let us place ourselves in a HAG context, as in the previous section. The notion of substack in Definition \ref{substack} is too strict, since it is not homotopy invariant. The following lemma shows that homotopy monomorphisms of stacks are homotopy invariant replacements of inclusions of substacks.

\begin{lem}\label{parastacks}
Given a morphism of stacks $f\colon F\r G$, the following statements are equivalent: 
\begin{enumerate}
\item $f(A)\colon F(A)\r G(A)$ is a homotopy monomorphism of simplicial sets  for any commutative algebra $A$ in $\C V$.

\item $f$ is a homotopy monomorphism in  $\operatorname{SPr}(\aff_{\C V})$.

\item $f$ is a homotopy monomorphism in $\aff_{\C V}^{\sim,\tau}$.
\end{enumerate}
\end{lem}

\begin{proof}
Homotopy limits in $\operatorname{SPr}(\aff_{\C V})$ are computed pointwise, hence $(1)\Leftrightarrow(2)$ follows. The equivalence $(2)\Leftrightarrow(3)$ is a consequence of the fact that the homotopy product $F\times^{h}_{G}F$ in $\aff_{\C V}^{\sim,\tau}$ coincides with the corresponding homotopy product in  $\operatorname{SPr}(\aff_{\C V})$. 
\end{proof}

Recall from the previous section that $\phi\colon\mathtt{Ass}\r \mathtt{uAss}$ induces morphism of stacks
\begin{equation}\label{laclau}
\algebrau{\phi}{\C V}\colon \algebrau{\mathtt{uAss}}{\C V}\To \algebrau{\mathtt{Ass}}{\C V}.
\end{equation}
Moreover, if $A$ is a commutative algebra and $M$ is an $A$-module, $\phi$ also induces a morphism of stacks
\begin{equation}\label{laclau2}
\phi^{*}\colon\rmapu_{\operad{\C V}}(\mathtt{uAss},\mathtt{End}_{\Mod{A}}(M))\To \rmapu_{\operad{\C V}}(\mathtt{Ass},\mathtt{End}_{\Mod{A}}(M)).
\end{equation}

\begin{prop}\label{qc}
The homotopy pull-back of the morphism $\algebrau{\phi}{\C V}$ along any map $g\colon \rspecu(A)\r\algebrau{\mathtt{Ass}}{\C V}$ represented by an associative algebra with underlying perfect $A$-module $M$ is an affine stack.
\end{prop}

\begin{proof}
The homotopy pull-back of $\xi^{\mathtt{Ass}}\colon \algebrau{\mathtt{Ass}}{\C V}\r\qcoh$ along the composite $\xi^{\mathtt{Ass}}g$ is the affine stack $\rmapu_{\operad{\C V}}(\mathtt{Ass},\mathtt{End}_{\Mod{A}}(M))$. Moreover, the homotopy pull-back of $\xi^{\mathtt{uAss}}$ along $\xi^{\mathtt{Ass}}g$ is the affine stack $\rmapu_{\operad{\C V}}(\mathtt{uAss},\mathtt{End}_{\Mod{A}}(M))$. Hence, 
\begin{equation*}
\xymatrix@C=30pt{
\rmapu_{\operad{\C V}}(\mathtt{uAss},\mathtt{End}_{\Mod{A }}(M ))\ar[r]^{\eqref{laclau2}}\ar[d]&\rmapu_{\operad{\C V}}(\mathtt{Ass},\mathtt{End}_{\Mod{A }}(M ))\ar[d]\\
\algebrau{\mathtt{uAss}}{\C V}\ar[r]_-{\algebrau{\phi}{\C V}}&\algebrau{\mathtt{Ass}}{\C V}}
\end{equation*}
is a homotopy pull-back, and the homotopy pull-back of $\algebrau{\phi}{\C V}$ along $g$ coincides with the homotopy pull-back of \eqref{laclau2} along the morphism $\rspecu(A)\r\rmapu_{\operad{\C V}}(\mathtt{Ass},\mathtt{End}_{\Mod{A }}(M ))$ defined by the universal property of a homotopy pull-back. Now, the result follows from the fact that a homotopy pull-back of affine stack is affine.
\end{proof}

Assume from now on that the base category $\C V$ of our HAG context is simplicial or complicial.

\begin{thm}\label{laclauthm}
The morphisms  \eqref{laclau} and \eqref{laclau2} are homotopy monomorphisms of stacks.
\end{thm}

This result follows from Lemma \ref{parastacks}, Proposition \ref{momomo}, and the fact that \eqref{momo} is a homotopy monomorphism.

\begin{defn}[{\cite[Definition 1.2.3.1]{hagII}}]
An object $X$ in a model category $\C M$ is \emph{homotopically finitely presented} if the mapping space functor $\rmap_{\C M}(X,-)$ preserves filtered homotopy colimits. A morphism $X\r Y$ in $\C M$ is \emph{homotopically finitely presented} if it is homotopically finitely presented as an object in the comma model category $X\downarrow\C M$.
\end{defn}

Retracts of homotopically finitely presented objects are also homotopically finitely presented.

\begin{thm}\label{esfp}
Assume that $\C V$ is compactly generated in the sense of \cite[Definition 1.2.3.4 (3)]{hagII} and that the tensor unit is homotopically finitely presented.
Suppose $M$ is a perfect $A$-module. Then the morphism of affine stacks \eqref{laclau2} is represented by a homotopically finitely presented morphism of commutative $A$-algebras.
\end{thm}

Assuming that $\C V$ is compactly generated is not a very strong hypothesis. It is satisfied by any model category which is locally finitely presentable \cite[Definition 1.9]{adamekrosicky} and finitely generated \cite[Definition 2.1.17]{hmc}. All the underlying model categories of the HAG contexts considered in \cite{hagII} satisfy these properties. Neither is very strong to assume that the tensor unit is homotopically finitely presented. This assumption is also used in \cite{hagII} to show that  the general linear group is a Zariski open affine substack of the affine stack of matrices in a HAG context, see \cite[Propositions 1.2.9.4 and 1.3.7.10]{hagII}.

In the proof of Theorem \ref{esfp}, we use the following proposition.

\begin{prop}\label{retracto2}
For any commutative algebra $A$ and any $A$-module $M$, the map
$$\psi^{*}\colon
\rmapu_{\operad{\C V}}(\mathtt{uAss},\mathtt{End}_{\Mod{A}}(M))\To
\rmapu_{\operad{\C V}}(\mathtt{Ass}^{\operatorname{qu}},\mathtt{End}_{\Mod{A}}(M))$$
admits a retraction in the homotopy category of stacks.
\end{prop}

This result follows in the same way as Proposition \ref{retracto}, using Remark \ref{vayavaya}.

\begin{proof}[Proof of Theorem \ref{esfp}]
We can suppose without loss of generality that $A$ and $M$ are fibrant and cofibrant. 
By Definition \ref{assqu}, Remark \ref{yayaya}, and the proof of Theorem \ref{lageo}, if $B$ is a fibrant and cofibrant commutative $A$-algebra representing $\rmapu_{\operad{\C V}}(\mathtt{Ass},\mathtt{End}_{\Mod{A}}(M))$ then a commutative $A$-algebra $D$ representing $\rmapu_{\operad{\C V}}(\mathtt{Ass}^{\operatorname{qu}},\mathtt{End}_{\Mod{A}}(M))$ can be defined as a certain homotopy push-out
$$\xymatrix{
{\begin{array}{c} \operatorname{Sym}_{A}((\tilde\unit\coprod\tilde\unit)\otimes M\otimes_{A} M^{\vee})\\\coprod\\
  \operatorname{Sym}_{A}((\tilde\unit\coprod\tilde\unit)\otimes M\otimes_{A} M^{\vee})
  \end{array}}
  \ar[d]\ar[r]\ar@{}[rd]|{\text{htpy. push}}&
{\begin{array}{c} \operatorname{Sym}_{A}(C\otimes M\otimes_{A} M^{\vee})\\\coprod\\
  \operatorname{Sym}_{A}(C\otimes M\otimes_{A} M^{\vee})
  \end{array}}\ar[d]\\
B\coprod \operatorname{Sym}_{A}(\tilde\unit\otimes M^{\vee})\ar[r]&D}$$
in the category of commutative $A$-algebras. 
Here $\operatorname{Sym}_{A}$ denotes the free commutative $A$-algebra functor.

Since $\C V$ is compactly generated, derived change of coefficient functors along commutative algebra morphisms preserve homotopically finitely presented objects by \cite[Proposition 1.2.3.5]{hagII}. 
Since $\unit$ is homotopically finitely presented in $\C V$, then so is $A$ in $\Mod{A}$ and more generally any perfect $A$-module, see \cite[Proposition 1.2.3.7]{hagII}. Therefore all free $A$-algebras in the previous homotopy push-out are homotopically finitely presented. Hence, the map $B\r D$ is homotopically  finitely presented.

By Proposition \ref{retracto2}, any $B$-algebra representing $\rmapu_{\operad{\C V}}(\mathtt{uAss},\mathtt{End}_{\Mod{A}}(M))$ is a homotopy retract of $D$, hence also homotopically  finitely presented.
\end{proof}

\begin{cor}
The homotopy pull-back of $\algebrau{\phi}{\C V}$ along any  $g\colon \rspecu(A)\r\algebrau{\mathtt{Ass}}{\C V}$ represented by an associative algebra with underlying perfect $A$-module $M$ is an affine stack represented by a homotopically finitely presented $A$-algebra.
\end{cor}

\begin{proof}
In the proof of Proposition \ref{qc} we showed that the homotopy pull-back in the statement can be obtained as a homotopy pull-back along \eqref{laclau2}. Hence this result follows from Theorem \ref{esfp} and the fact that the cobase change of a homotopically finitely morphism is homotopically finitely presented \cite[Proposition 1.2.3.3 (3)]{hagII}.
\end{proof}

We end this section with an example showing that the affine stack of associative algebra structures $\rmapu_{\operad{\C V}}(\mathtt{Ass},\mathtt{End}_{\Mod{A}}(M))$ need not be homotopically finitely presented over $\rspecu(\unit)$, even for $M$ perfect. This is a big contrast with the classical situation.

Let us place ourselves in any of the two complicial algebraic geometry contexts in \cite[\S 2.3]{hagII}. We will use homological notation for complexes $X_{*}$, so differentials have degree $-1$, $d\colon X_{n}\r X_{n-1}$, i.e.~our complexes are \emph{chain} complexes, not \emph{cochain} complexes. We write $|x|=n$ if $x\in X_{n}$.

Consider $\Bbbk=\mathbb Q$, $A=\mathbb Q$ and $M=\Sigma^{m}\mathbb Q$ the $m$-fold suspension of $\mathbb Q$. Denote $e$ the degree $m$ generator of $\Sigma^{m}\mathbb Q$. In order to compute a commutative differential graded algebra (CDGA) representing the affine stack $\rmapu_{\operad{\operatorname{Ch}(\mathbb Q)}}(\mathtt{Ass},\mathtt{End}_{\operatorname{Ch}(\mathbb Q)}(\Sigma^{m}\mathbb Q))$ we need a cofibrant resolution of $\mathtt{Ass}$ in $\operad{\operatorname{Ch}(\mathbb Q)}$. The following  operad is  a well-known cofibrant resolution.

\begin{defn}
The \emph{differential graded $A$-infinity operad} $\mathtt{A}_{\infty}$ is defined as follows. The underlying graded operad of $\mathtt{A}_{\infty}$ is freely generated by
$$\mu_{n}\in \mathtt{A}_{\infty}(n)_{n-2},\qquad n\geq 2,$$
and the differential is given by 
\begin{equation*}\label{dm}d(\mu_{n})=\sum_{
\begin{array}{c}\\[-16pt]
\scriptstyle p+q-1=n\\[-1.5mm]
\scriptstyle 1\leq i\leq p
\end{array}
} 
(-1)^{
q
p
+(q-1)i
}
\mu_{p}\circ_{i}\mu_{q}.
\end{equation*}
\end{defn}

The operad $\mathtt{A}_{\infty}$ is Stasheff's operad \cite{hahs}. It is fibrant and cofibrant, and the morphism $\mathtt{A}_{\infty}\r\mathtt{Ass}$, defined by $\mu_{2}\mapsto 1\in\mathbb Q=\mathtt{Ass}(2)$ and $\mu_{n}\mapsto 0$, $n>2$, is a trivial fibration.

 Let $A$ be any  CDGA. An $\mathtt{A}_{\infty}$-algebra structure on the $A$-module $\Sigma^{m}\mathbb Q\otimes_{\mathbb Q}A\cong\Sigma^{m}A$ is determined by degree $n-2$ morphisms of graded $\mathbb Q$-modules,
$$\mu_{n}\colon \Sigma^{m}\mathbb Q\otimes_{\mathbb Q}\st{n}\cdots\otimes_{\mathbb Q}\Sigma^{m}\mathbb Q\To \Sigma^{m}\mathbb Q\otimes_{\mathbb Q}A,\qquad n\geq2,$$
satisfying certain equations. These morphisms  are determined by the structure constants in $A$,
\begin{align}\label{constants}
x_{n},&&|x_{n}|&=n-2+mn-m,& n&\geq 2,
\end{align}
such that
$$\mu_{n}(e,\st{n}\dots, e)=x_{n}e.$$

One can straightforwardly check that a choice of structure constants determines an $\mathtt{A}_{\infty}$-algebra structure if and only if the following equations hold in $A$, 
\begin{align}\label{constantsdif}
d(x_{n})&=\sum_{
\begin{array}{c}\\[-16pt]
\scriptstyle p+q-1=n\\[-1.8mm]
\scriptstyle 1\leq i\leq p
\end{array}
}  (-1)^{(q-1)i+(p+m(i-1))(q+1)m}x_{q}x_{p},&n&\geq 2.
\end{align}
Hence we deduce the following result.

\begin{prop}
The stack $\rmapu_{\operad{\operatorname{Ch}(\mathbb Q)}}(\mathtt{Ass},\mathtt{End}_{\operatorname{Ch}(\mathbb Q)}(\Sigma^{m}\mathbb Q))$ is isomorphic to $\rspecu(B_{m})$ in the homotopy category of stacks, where $B_{m}$ is the CDGA whose underlying graded commutative algebra is freely generated by the symbols \eqref{constants} and such that the differential is defined by \eqref{constantsdif}.
\end{prop}

\begin{prop}\label{hfp}
For $m\leq -2$, the CDGA $B_{m}$ is not homotopically finitely presented. 
\end{prop}

\begin{proof}
In this range, $B_{m}$ is a minimal Sullivan algebra in the sense of \cite[\S II.12]{rhtfht}. Let $B_{m,r}\subset B_{m}$, $r\geq 2$, be the subalgebra generated by the $x_{n}$ with $n\leq r$. Notice that the differential of $B_{m} $ restricts to $B_{m,r} $, so $B_{m,r} $ is actually a sub-CDGA of $B_{m}$. These sub-CDGA are also minimal Sullivan algebras. They define an increasing filtration of $B_{m} $ such that $$B_{m} =\bigcup_{r\geq 2} B_{m,r} =\colim_{r} B_{m,r}=\operatorname{hocolim}_{r} B_{m,r} .$$ This is indeed  a homotopy colimit  since the inclusions $B_{m,r}\subset B_{m,r+1}$ are cofibrations of CDGAs. 

If $B_{m}$ were homotopically finitely presented, the identity in $B_{m}$ would factor up to homotopy through the inclusion of a certain $B_{m,r}\subset B$. This factorization would induce a quasi-isomorphism from the linear part of $B_{m}$ to itself by \cite[Proposition II.14.13]{rhtfht}.  This is impossible, since the linear part of $B_{m}$ is the unbounded graded $\mathbb Q$-module  with basis $\{x_{n}\}_{n\geq 2}$ and trivial differential, and the linear part of  $B_{m,r}$ is bounded. It is actually the subcomplex spanned by $\{x_{n}\}_{r\geq n\geq 2}$.
\end{proof}

The stack $\underline{\operatorname{Vect}}_{n}$ of rank $n$ vector bundles, $n\geq 0$, is $1$-geometric in the complicial algebraic geometry contexts since they satisfy the assumptions recalled in Example \ref{1geo1} (1). We consider the substack $\underline{\operatorname{Vect}}_{n}[m]\subset\qcoh$, $m\in \mathbb Z$, such that, for any CDGA $A$, the connected components of $\underline{\operatorname{Vect}}_{n}[m](A)$ correspond to the $m$-fold suspensions of rank $n$ vector bundles. Obviously, suspension defines an isomorphism $\underline{\operatorname{Vect}}_{n}[m]\cong \underline{\operatorname{Vect}}_{n}$, $m\in\mathbb Z$, in the homotopy category of stacks, hence $\underline{\operatorname{Vect}}_{n}[m]$ is $1$-geometric. Notice, however, that this isomorphism is not compatible with the inclusions into $\qcoh$. 

\begin{cor}
For $m\leq -2$, the affine morphism between $1$-geometric stacks $\xi^{\mathtt{Ass}}_{\underline{\operatorname{Vect}}_{1}[m]}\colon \algebrau{\mathtt{Ass}}{\underline{\operatorname{Vect}}_{1}[m]}\r\underline{\operatorname{Vect}}_{1}[m]$  is not categorically locally finitely presented  in the sense of \cite[Definition 1.3.6.4 (1)]{hagII}.
\end{cor}

\begin{proof}
This follows from the fact that the homotopy pull-back of $\xi^{\mathtt{Ass}}_{\underline{\operatorname{Vect}}_{1}[m]}$ along the essentially unique morphism $\rspecu(\mathbb Q)\r \underline{\operatorname{Vect}}_{1}[m]$, represented by $\Sigma^{m}\mathbb Q$, is the affine stack $\rmapu_{\operad{\operatorname{Ch}(\mathbb Q)}}(\mathtt{Ass},\mathtt{End}_{\operatorname{Ch}(\mathbb Q)}(\Sigma^{m}\mathbb Q))$, which is not homotopically finitely presented.
\end{proof}

%
%

\appendix

\section{The very strong unit axiom}

In this short appendix we consider a strengthening of Hovey's unit axiom and its consequences. Here, unlike in Definition \ref{nota1} and the rest of the paper, monoidal model categories are not assumed to satisfy the strong unit axiom, but just Hovey's unit axiom.

\begin{defn}\label{vsua}
A monoidal model category $\C C$ satisfies the \emph{very strong unit axiom} if for any object $X$ and any cofibrant replacement $q\colon\tilde\unit\st{\sim}\onto\unit$ of the tensor unit, the morphisms $X\otimes q$ and $q\otimes X$ are weak equivalences.
\end{defn}

This axiom has been implicitly used  in the proof of \cite[Proposition 1.15]{htapm}. The name `very strong unit axiom' was coined by Berger in private communication.

\begin{rem}\label{sua}
If the very strong unit axiom holds for a certain cofibrant resolution of $\unit$ then it holds for any cofibrant resolution of $\unit$. In particular, it is satisfied when $\unit$ is cofibrant.

The \emph{strong unit axiom} considered in \cite[Definition A.9]{htnso2} is a weaker version where $X$ runs only over the pseudo-cofibrant objects. Recall that $X$ is \emph{pseudo-cofibrant} if the functors $X\otimes -$ and $-\otimes X$ preserve cofibrations. This is equivalent to say that $X\otimes -$ and $-\otimes X$ are left Quillen functors, see \cite[Remark A.2]{htnso2}.
Cofibrant objects are pseudo-cofibrant by the push-out product axiom. The paradigm of pseudo-cofibrant object which need not be cofibrant is the tensor unit $\unit$. Moreover, objects $X$ for which there exists a cofibration $\unit\into X$ are also pseudo-cofibrant, see \cite[Remark B.2]{htnso2}. These objects are called \emph{$\unit$-cofibrant}.
\end{rem}

The following three lemmas can be proved as \cite[Lemmas A.11, A.12 and A.13]{htnso2}, respectively, using the very strong unit axiom instead of the strong unit axiom.

\begin{lem}\label{davidwhite}
Suppose that, for $q\colon\tilde\unit\st{\sim}\onto\unit$ a cofibrant resolution of the tensor unit, the functors $\tilde\unit\otimes -$ and $-\otimes\tilde\unit$ preserve weak equivalences.  Then the very strong unit axiom holds.
\end{lem}

\begin{rem}\label{setiene}
In particular, the very strong unit axiom holds in monoidal model categories where cofibrant objects are \emph{flat}. Recall that this means that the functors $X\otimes-$ and $-\otimes X$ preserve weak equivalences if $X$ is cofibrant.
\end{rem}

The converse  is also true, even something stronger holds.

\begin{lem}\label{chari}
If the very strong unit axiom holds, a morphism $f\colon U\r V$ is a weak equivalence if and only if $f\otimes\tilde\unit$ is a weak equivalence for some cofibrant replacement $\tilde\unit$ of the tensor unit. The same is true replacing $f\otimes\tilde\unit$ with $\tilde\unit\otimes f$.
\end{lem}

The following result is used to show that module categories in a HAG context are algebras over the base symmetric monoidal model category.

\begin{lem}\label{pseudowe}
Let $\C{C}$ be a monoidal model category satisfying the very strong unit axiom.
If $f\colon U\st{\sim}\r V$ is a weak equivalence  with pseudo-cofibrant source and target and $X$ is any object, then $f\otimes X$ and $X\otimes f$ are weak equivalences.
\end{lem}


\end{document}